\documentclass[11pt,reqno]{amsart}

\usepackage[T1]{fontenc}
\usepackage[utf8]{inputenc}
\usepackage{amsmath}
\usepackage{amssymb}
\usepackage{mathtools}
\usepackage{microtype}
\usepackage[textwidth=14.5cm,textheight=22.5cm,centering]{geometry}
\usepackage{xcolor}
\usepackage{needspace}
\usepackage{tikz}
\usetikzlibrary{arrows.meta,calc,positioning,decorations.markings}
\usepackage[
  colorlinks=true,
  linkcolor=blue!55!black,
  citecolor=green!40!black,
  urlcolor=blue!65!black,
  pdftitle={Signed Dual Graphs and Face Spectral Gaps of Cayley 2-Complexes},
  pdfauthor={Bobo Hua; Yong Lin; Tao Wang; Haohang Zhang},
  pdfsubject={Face and coexact spectral gaps of Cayley 2-complexes},
  pdfkeywords={Cayley 2-complex; signed dual graph; face Laplacian; spectral gap; amenability}
]{hyperref}

\newtheorem{theorem}{Theorem}[section]
\newtheorem{proposition}[theorem]{Proposition}
\newtheorem{lemma}[theorem]{Lemma}
\newtheorem{corollary}[theorem]{Corollary}
\newtheorem{conjecture}[theorem]{Conjecture}
\theoremstyle{definition}
\newtheorem{definition}[theorem]{Definition}
\theoremstyle{remark}
\newtheorem{remark}[theorem]{Remark}

\DeclareMathOperator{\spec}{spec}
\DeclareMathOperator{\ran}{ran}
\newcommand{\cP}{\mathcal P}
\newcommand{\cQ}{\mathcal Q}
\newcommand{\cF}{\mathcal F}
\newcommand{\cH}{\mathcal H}

\newcommand{\C}{\mathbb C}
\newcommand{\R}{\mathbb R}
\newcommand{\Z}{\mathbb Z}

\title[Signed dual graphs and face spectral gaps]{Signed Dual Graphs and Face Spectral Gaps of Cayley 2-Complexes}

\author{Bobo Hua}
\address{School of Mathematical Sciences, LMNS, Fudan University, Shanghai 200433, China}
\email{bobohua@fudan.edu.cn}

\author{Yong Lin}
\address{Yau Mathematical Sciences Center and Department of Mathematical Sciences, Tsinghua University, Beijing 100084, China}
\email{yonglin@tsinghua.edu.cn}

\author{Tao Wang}
\address{Beijing International Center for Mathematical Research, Peking University, Beijing 100871, China}
\email{taowang25@pku.edu.cn}

\author{Haohang Zhang}
\address{Yau Mathematical Sciences Center and Department of Mathematical Sciences, Tsinghua University, Beijing 100084, China; Department of Mathematics and Statistics, P.O. Box 68 (Pietari Kalmin katu 5), FI-00014 University of Helsinki, Finland}
\email{haohang.x.zhang@helsinki.fi}

\date{\today}



\begin{document}

\begin{abstract}
We construct a weighted signed dual multigraph for the Cayley \(2\)-complex of a finite presentation and represent its \(\ell^2\) face Laplacian as a signed graph Laplacian with a diagonal potential. We prove that the face spectral gap above the kernel vanishes for the Baumslag--Solitar groups \(\operatorname{BS}(1,m)\), \(m\ge2\), using explicit finitely supported face chains. Our main vanishing theorem gives the same conclusion for every finitely presented amenable group with ordinary first Betti number at least two. 
\end{abstract}

\maketitle

\raggedbottom
\section{Introduction}
\label{sec:introduction}

The interplay between amenability, spectral gaps, and volume growth has long been a central theme in geometric group theory. F\o lner's criterion characterizes amenability through nonempty finite sets with arbitrarily small ratios of boundary size to volume \cite{Folner1955}. Combined with Dodziuk's discrete Cheeger inequality \cite{Dodziuk1984}, it explains Kesten's characterization: the \(\ell^2\) spectral bottom of a Cayley graph Laplacian vanishes precisely when the group is amenable \cite{Kesten1959,KestenMean1959}. Brooks proved the analogous result for universal covers of closed Riemannian manifolds \cite{Brooks1981}. A positive spectral bottom yields a positive isoperimetric constant and hence exponential growth. Amenable groups can also grow exponentially, so vanishing of the vertex spectral bottom cannot distinguish polynomial from exponential growth, motivating the study of higher degree Laplacians.

Dubashinskiy studied the relation between volume growth and the coexact \(1\)-Laplacian for infinite finitely presented groups, proving that a positive spectral gap forces exponential growth unless the group is virtually infinite cyclic \cite[Theorem~1.5]{Dubashinskiy2025}. In the amenable setting, where the vertex spectral bottom vanishes, he asked whether the Baumslag--Solitar groups \(\operatorname{BS}(1,m)\), \(m\ge2\), have a positive coexact gap \cite[p.~4]{Dubashinskiy2025}. 

We approach this question through the face Laplacian of the Cayley \(2\)-complex \(X=X_{\cP}\) associated with a finite presentation \(\cP=\langle S\mid R\rangle\) of \(\Gamma\) \cite{LyndonSchupp1977}. The face and coexact \(1\)-Laplacians on \(X\) have the same nonzero spectrum; comparison with Dubashinskiy's model is discussed in Section~\ref{sec:bs}. We construct a weighted signed dual multigraph \(X^\vee\) whose vertices are the faces of \(X\), indexed by \(\Gamma\times R\). Its signed graph Laplacian \(L_\sigma^\vee\), together with a diagonal potential \(V\), represents the face Laplacian.

Write \(\partial_2\) for the cellular boundary from faces to edges and \(\Delta_k\) for the unnormalized \(\ell^2\)-cellular Laplacians; precise definitions are given in Section~\ref{sec:cayley}. For a nonnegative self-adjoint operator \(A\), let \(\spec_+(A)=\spec(A)\setminus\{0\}\), and set
\[
\begin{aligned}
  \lambda_0(X)&=\inf\spec(\Delta_0), &
  \lambda_2^+(X)&=\inf\spec_+(\Delta_2),\\
  \lambda_{1,\mathrm{coex}}(X)
    &=\inf\spec\!\left(\Delta_1|_{\overline{\ran\partial_2}}\right).
\end{aligned}
\]
The first two quantities are the \emph{vertex gap} and \emph{face gap}; the third is the \emph{coexact gap}. For infinite \(\Gamma\), degree zero has no nonzero \(\ell^2\) constant mode, whereas degree two may contain \(\ell^2\) face chains whose boundaries cancel. This is why the face gap removes zero before taking the spectral bottom. Our first theorem identifies the face operator on the dual graph and relates its spectrum to that of the \(1\)-Laplacian.

\begin{theorem}
\label{thm:main}
Let \(\cP=\langle S\mid R\rangle\) be a finite presentation of \(\Gamma\). There is a weighted signed dual multigraph \(X^\vee\) on \(\Gamma\times R\) and a diagonal potential \(V\) such that
\[
  \Delta_2=L_\sigma^\vee+V.
\]
The nonzero spectrum of the \(1\)-Laplacian satisfies
\begin{equation*}
  \spec_+(\Delta_1)
  =\spec_+(\Delta_0)\cup\spec_+(L_\sigma^\vee+V).
\end{equation*}
Moreover,
\[
  \lambda_{1,\mathrm{coex}}(X)
  =\lambda_2^+(X)
  =\inf\spec_+(L_\sigma^\vee+V).
\]
\end{theorem}

The theorem allows us to study the face and coexact gaps together. Although the dual graph and numerical gap values depend on the presentation, vanishing of either gap is independent of the chosen finite presentation; see Proposition~\ref{prop:presentation-independence}. We first consider the amenable family
\[
  \operatorname{BS}(1,m)
  =\langle a,b\mid aba^{-1}=b^m\rangle,\qquad m\ge2,
\]
within the Baumslag--Solitar groups introduced in \cite{BaumslagSolitar1962}. By constructing finite face chains with explicitly controlled boundaries, we prove that their face spectral gaps vanish.

\begin{theorem}
\label{thm:bs-gap}
Let \(X_m\) be the Cayley \(2\)-complex of the standard presentation of \(\operatorname{BS}(1,m)\), \(m\ge2\). Then
\[
  \ker\partial_2=\{0\},\qquad
  \lambda_2^+(X_m)=\lambda_{1,\mathrm{coex}}(X_m)=0.
\]
\end{theorem}

Corollary~\ref{cor:bs-bounded-cycle} transfers this vanishing to Dubashinskiy's bounded cycle operator and answers his question for this family. The vanishing can also be deduced from earlier homological results, as explained in Remark~\ref{rem:bs-homology}; our proof gives an explicit sequence of finite face chains and also establishes \(\ker\partial_2=\{0\}\).

We next prove a vanishing theorem for amenable groups with \(b_1(\Gamma)\ge2\). Here \(b_1(\Gamma)=\dim_{\R}H_1(\Gamma;\R)\), with \(H_1(\Gamma;\R)\) denoting group homology with trivial real coefficients. For the presentation complex \(K_{\cP}\), the Hurewicz and universal coefficient theorems \cite[Theorems~2A.1 and~3A.3]{Hatcher2002}, together with the group homology formula \cite[Theorem~6.1.12]{Weibel1994}, give natural isomorphisms
\[
 H_1(\Gamma;\R)
 \cong H_1(K_{\cP};\R)
 \cong (\Gamma/[\Gamma,\Gamma])\otimes_{\Z}\R.
\]
We use ordinary homology \(H_1(K_{\cP};\R)\) to interpret \(H_1(\Gamma;\R)\) and \(b_1(\Gamma)\) below.

\begin{theorem}
\label{thm:betti-two-gap}
Let \(\Gamma\) be an infinite amenable finitely presented group with \(b_1(\Gamma)\ge2\). For every finite presentation of \(\Gamma\), the associated Cayley \(2\)-complex \(X\) satisfies
\[
  \lambda_{1,\mathrm{coex}}(X)=\lambda_2^+(X)=0.
\]
\end{theorem}

We initially sought to detect exponential growth in amenable groups through a positive face gap, since their vertex gap always vanishes. Our vanishing results, together with the examples in Section~\ref{sec:examples}, instead suggest that, outside the virtually cyclic class, a positive face gap may force nonamenability. We therefore propose the following conjecture.

\begin{conjecture}
\label{conj:coexact-implies-vertex}
Let \(\Gamma\) be an infinite finitely presented group that is not virtually cyclic, and let \(X\) be the Cayley \(2\)-complex of any finite presentation of \(\Gamma\). Then
\[
  \lambda_2^+(X)>0
  \quad\Longrightarrow\quad
  \lambda_0(X)>0.
\]
Equivalently, every amenable group in this class has zero face and coexact gaps.
\end{conjecture}

Sections~\ref{sec:cayley}--\ref{sec:spectral-transfer} build the operator framework, passing from the cellular setup to the signed dual representation of the face Laplacian, then to spectral transfer and presentation independence of gap vanishing. We apply this framework in Section~\ref{sec:bs} to answer Dubashinskiy's question for the Baumslag--Solitar groups. The homological and product arguments in Section~\ref{sec:amenable} extend the vanishing results to amenable groups with \(b_1\ge2\) and amenable products. The examples in Section~\ref{sec:examples} clarify the scope of the conjecture.

\section{Cellular setup}
\label{sec:cayley}

Let \(K_{\cP}\) be the presentation complex of \(\cP=\langle S\mid R\rangle\). Here \(R\) is an indexed family of words; repeated words label distinct cells. The complex has one vertex \(v\), one oriented edge for each \(s\in S\), and one oriented \(2\)-cell \(c_r\) attached along each indexed relator \(r\). Its universal cover, the simply connected covering space, is the Cayley \(2\)-complex.

For each \(r\), choose a map
\[
  \phi_r:(D^2,1)\longrightarrow(K_{\cP},v)
\]
whose interior parametrizes \(c_r\) and whose boundary is its attaching loop. Such a map is called a characteristic map. Marking \(1\in\partial D^2\), we require the positively oriented boundary to start there and read \(r\). The standard orientation of \(D^2\) then induces the orientation of \(c_r\) via \(\phi_r\).

Let \(p:X=X_{\cP}\to K_{\cP}\) be the universal covering projection. Choose a lift of \(v\) and identify \(X^{(0)}\) with \(\Gamma\). Under this identification, deck transformations, the covering symmetries of \(p\), act by left multiplication, \(g\mapsto hg\) for \(h\in\Gamma\), while the Cayley graph edges are defined by right multiplication by generators: \(e(g,s):g\to gs\). Left and right multiplication commute, since \(h(gs)=(hg)s\). Hence the deck transformation associated with \(h\) sends \(e(g,s)\) to \(e(hg,s)\), preserving adjacency, edge labels, and orientations. Write \(\cF(X)\) for the set of open \(2\)-cells of \(X\).

\begin{definition}
For \(g\in\Gamma\) and \(r\in R\), let
\[
  \widetilde\phi_{g,r}:D^2\longrightarrow X
\]
be the unique lift of \(\phi_r\) satisfying \(\widetilde\phi_{g,r}(1)=g\). The corresponding open \(2\)-cell is
\[
  F(g,r):=\widetilde\phi_{g,r}(\mathring D^2).
\]
We orient \(F(g,r)\) by lifting this orientation of \(c_r\), so its positively oriented boundary reads \(r\) starting at \(g\).
\end{definition}

The face \(F(g,r)\) is uniquely determined by the vertex \(g\) and the indexed relator \(r\), independently of the word chosen to represent \(g\).

Figure~\ref{fig:lifted-face} shows the commutative diagram for \(\cP=\langle a,b,c,d\mid abc\rangle\). The map \(\phi_r\) sends the positively oriented boundary \(\partial D^2\) successively along the petal edges \(a,b,c\) in the indicated directions, and maps the interior \(\mathring D^2\) homeomorphically onto the open \(2\)-cell \(c_r\), represented schematically by the blue shading.

\begin{figure}[htbp]
\centering
\begin{tikzpicture}[
  >=Latex,
  every node/.style={font=\small},
  oriented boundary/.style={
    postaction={decorate},
    decoration={markings,mark=at position .25 with {
      \arrow{Latex[length=1.7mm,width=1.1mm]}
    }}
  },
  source cell/.style={
    draw=blue!55!black,fill=blue!10,line width=1pt
  },
  shaded petal/.style={
    draw=blue!55!black,fill=blue!10,line width=.9pt
  },
  plain petal/.style={
    draw=black!55,fill=white,line width=.6pt
  },
  lifted face/.style={fill=blue!10,draw=none},
  cycle edge/.style={
    draw=blue!55!black,line width=.9pt,
    -{Latex[length=1.7mm,width=1.1mm]},
    shorten <=1.7pt,shorten >=1.7pt
  },
  relation germ/.style={draw=blue!50!black,line width=.65pt},
  d germ/.style={draw=black!48,line width=.55pt},
  map arrow/.style={line width=.4pt,-{Latex[length=1.6mm,width=1mm]}},
  marked vertex/.style={circle,fill=black,inner sep=1.25pt}
]
  % The source disk.
  \coordinate (d) at (-2.00,-1.20);
  \coordinate (one) at (-1.34,-1.20);
  \draw[source cell,oriented boundary] (d) circle (.66);
  \node at (d) {\(D^2\)};

  % Four generator petals; only the letters traversed by r=abc are shaded.
  \coordinate (v) at (3.35,-1.20);
  \foreach \ang/\lab in {315/d}{
    \begin{scope}[shift={(v)},rotate=\ang]
      \path[plain petal,oriented boundary]
        (0,0)
        .. controls (.36,-.18) and (.83,-.28) .. (1.05,0)
        .. controls (.83,.28) and (.36,.18) .. (0,0);
      \node[font=\scriptsize,text=black!70] at (1.24,0) {\(\lab\)};
    \end{scope}
  }
  \foreach \ang/\lab in {45/a,135/b,225/c}{
    \begin{scope}[shift={(v)},rotate=\ang]
      \path[shaded petal,oriented boundary]
        (0,0)
        .. controls (.36,-.18) and (.83,-.28) .. (1.05,0)
        .. controls (.83,.28) and (.36,.18) .. (0,0);
      \node[font=\scriptsize,text=blue!55!black] at (1.24,0) {\(\lab\)};
    \end{scope}
  }
  \node[blue!55!black,font=\scriptsize] at (4.70,-1.20)
    {\(r=abc\)};
  \node at (3.35,-2.43) {\(K_{\cP}\)};

  % The counterclockwise embedded lift of the abc-cycle.
  \coordinate (g) at (3.35,.25);
  \coordinate (ga) at (4.50,1.40);
  \coordinate (gab) at (2.20,1.40);
  \path[lifted face] (g)--(ga)--(gab)--cycle;

  % Together with the two cycle sides, these germs give eight directions.
  \foreach \ang in {0,180,225,315}
    \draw[relation germ] (g)--++(\ang:.42);
  \foreach \ang in {90,270}
    \draw[d germ] (g)--++(\ang:.42);
  \foreach \ang in {0,45,135,315}
    \draw[relation germ] (ga)--++(\ang:.42);
  \foreach \ang in {90,270}
    \draw[d germ] (ga)--++(\ang:.42);
  \foreach \ang in {45,135,180,225}
    \draw[relation germ] (gab)--++(\ang:.42);
  \foreach \ang in {90,270}
    \draw[d germ] (gab)--++(\ang:.42);

  \draw[cycle edge] (g)--node[right=1pt,font=\scriptsize] {\(a\)} (ga);
  \draw[cycle edge] (ga)--node[above=2pt,font=\scriptsize] {\(b\)} (gab);
  \draw[cycle edge] (gab)--node[left=1pt,font=\scriptsize] {\(c\)} (g);
  \node at (3.35,.9) {\(F(g,r)\)};
  \node[align=center] at (3.35,2.30)
    {\(X_{\cP}\)\\[-2pt]{\scriptsize local lift}};

  % The characteristic map, its lift, and the covering projection.
  \draw[map arrow] (one)--node[above,pos=.54] {\(\phi_r\)} (v);
  \draw[map arrow] (one) --
    node[above left,pos=.50] {\(\widetilde\phi_{g,r}\)} (g);
  \draw[map arrow] (g) to[bend left=12] node[right,pos=.52] {\(p\)} (v);

  % Redraw marked points and vertex labels.
  \node[marked vertex] at (one) {};
  \node[below right=1pt,font=\scriptsize] at (one) {\(1\)};
  \node[marked vertex] at (g) {};
  \node[marked vertex] at (ga) {};
  \node[marked vertex] at (gab) {};
  \node[font=\scriptsize] at (3.25,.55) {\(g\)};
  \node[font=\scriptsize] at (4.0,1.2) {\(ga\)};
  \node[font=\scriptsize] at (2.8,1.2) {\(gab\)};
  \node[marked vertex] at (v) {};
  \node[right=2pt,font=\scriptsize] at (v) {\(v\)};
\end{tikzpicture}
\caption{Lifting the characteristic map for \(r=abc\) to identify the face \(F(g,r)\).}
\label{fig:lifted-face}
\end{figure}
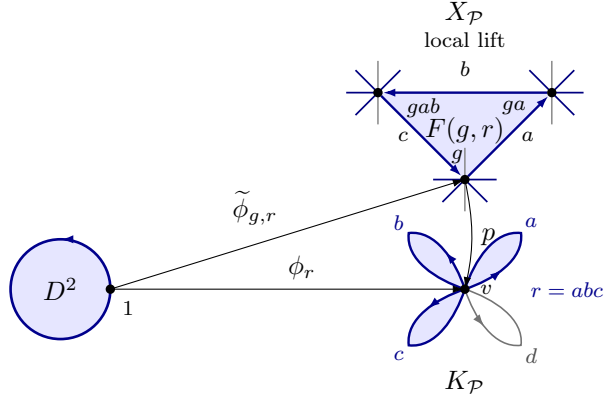

If
\[
  r=s_1^{\varepsilon_1}\cdots s_m^{\varepsilon_m}\in R,
  \qquad \varepsilon_j\in\{\pm1\},
\]
then the boundary path of \(F(g,r)\) starts at \(g_0=g\) and visits
\[
  g_j=gs_1^{\varepsilon_1}\cdots s_j^{\varepsilon_j}
  \quad(1\le j\le m).
\]
Since \(r=1\) in \(\Gamma\), we have \(g_m=g\). The \(j\)-th segment traverses \(e(g_{j-1},s_j)\) forward when \(\varepsilon_j=1\), and \(e(g_j,s_j)\) backward when \(\varepsilon_j=-1\).

By uniqueness of lifts,
\begin{equation}
\label{eq:face-parametrization}
  \cF(X)=\{F(g,r):g\in\Gamma,\ r\in R\},
  \qquad hF(g,r)=F(hg,r),
\end{equation}
with every face occurring once. Each \(r\in R\) labels a free \(\Gamma\)-orbit, independently of orientations.

Chains and Hilbert spaces are real unless complex test functions are specified. Let \(C_k^{(2)}(X)\) be the Hilbert space of square-summable cellular \(k\)-chains, with the chosen oriented cells as an orthonormal basis. The parametrizations of vertices, edges, and faces above give the isometric identifications
\[
  C_0^{(2)}(X)\cong\ell^2(\Gamma),\qquad
  C_1^{(2)}(X)\cong\ell^2(\Gamma)^{|S|},\qquad
  C_2^{(2)}(X)\cong\ell^2(\Gamma)^{|R|}.
\]
Thus a face chain is a tuple \((f_r)_{r\in R}\), where \(f_r(g)\) is the coefficient of the face \(F(g,r)\). Write \(\partial_k\) for the cellular boundary map on finitely supported chains.

For an oriented edge \(e\) of \(X\) and an oriented face \(F\), set
\begin{equation*}
  \eta(e,F):=\langle e,\partial_2F\rangle\in\Z.
\end{equation*}
Here \(\eta(e,F)\) counts forward minus backward traversals of \(e\) by the attaching loop of \(F\). We call \(F\) \emph{incident} with \(e\) when \(\eta(e,F)\ne0\). Opposite traversals cancel before forming the Laplacian, while repeated traversals in one direction retain their multiplicity in the boundary chain.

Deck transformations preserve the lifted cell orientations. Thus \(\partial_2\) is \(\Gamma\)-equivariant, and for \(h\in\Gamma\),
\begin{equation}
\label{eq:boundary-equivariance}
  \partial_2(hF)=h(\partial_2F),\qquad
  \eta(he,hF)=\eta(e,F).
\end{equation}

The finite generating set and total relator length bound the absolute row and column sums of the boundary matrices. Cauchy--Schwarz therefore gives a unique bounded extension \(\partial_k:C_k^{(2)}(X)\to C_{k-1}^{(2)}(X)\), \(k=1,2\). Since each relator boundary is closed, \(\partial_1\partial_2=0\), and these extensions form the chain complex
\[
  0\longleftarrow C_0^{(2)}(X)
  \xleftarrow{\partial_1}C_1^{(2)}(X)
  \xleftarrow{\partial_2}C_2^{(2)}(X)
  \longleftarrow0.
\]

Write \(\partial_k^*\) for the Hilbert space adjoint. The Hodge Laplacians are defined by
\begin{equation*}
  \Delta_0=\partial_1\partial_1^*,\qquad
  \Delta_1=\partial_1^*\partial_1+\partial_2\partial_2^*,\qquad
  \Delta_2=\partial_2^*\partial_2.
\end{equation*}
They are bounded and self-adjoint. Let \(E(X)=\{e(g,s):g\in\Gamma,\ s\in S\}\) be the oriented edge set, and let \(\delta_F\) be the face chain with coefficient \(1\) at \(F\) and \(0\) elsewhere. Then, for faces \(F,G\),
\begin{equation}
\label{eq:matrix-coefficients}
  \langle\delta_F,\Delta_2\delta_G\rangle
  =\langle\partial_2F,\partial_2G\rangle
  =\sum_{e\in E(X)}\eta(e,F)\eta(e,G).
\end{equation}
In particular,
\begin{equation}
\label{eq:q-def}
  q(F):=\|\partial_2F\|_2^2
  =\sum_{e\in E(X)}|\eta(e,F)|^2.
\end{equation}
The function \(q\) is \(\Gamma\)-invariant. If all nonzero incidences are \(\pm1\), then \(q(F)\) counts the edges with nonzero coefficient in \(\partial_2F\).

\section{The signed dual operator}
\label{sec:dual-graph}

We now construct a weighted signed dual multigraph from the incidences between edges and faces of \(X\). The construction records the individual terms in \eqref{eq:matrix-coefficients}, allowing us to express \(\Delta_2\) as a signed graph Laplacian \cite{Zaslavsky1982} plus a diagonal potential. Related signed Schr\"odinger representations for weighted simplicial complexes are given in \cite[Section~3.4]{BartmannKeller2025} and \cite[Theorem~6.3]{EnnaceurJadlaoui2026}.

\begin{definition}
\label{def:dual-graph}
The \emph{signed dual multigraph} \(X^\vee\) associated with \(X\) has one vertex for each face, so its vertex set is
\[
  V(X^\vee):=\cF(X).
\]
For every chosen oriented edge \(e\in E(X)\) and every unordered pair of distinct faces \(\{F,G\}\) satisfying \(\eta(e,F)\eta(e,G)\ne0\), introduce one dual edge \(\alpha=\alpha(e;F,G)\) joining \(F\) and \(G\). Its weight and sign are
\[
  \begin{aligned}
  w(\alpha)&:=|\eta(e,F)\eta(e,G)|,\\
  \sigma(\alpha)&:=-\operatorname{sgn}
    \bigl(\eta(e,F)\eta(e,G)\bigr).
  \end{aligned}
\]
In a coherently oriented surface tessellation, the minus sign makes dual edges positive because adjacent faces have opposite incidence signs. An edge incident with \(k\) faces produces \(\binom{k}{2}\) dual edges. If several edges of \(X\) are incident with the same pair of faces, they give distinct dual edges with the same endpoints, called \emph{parallel edges}.
\end{definition}

Nonunit weights occur already in \(\langle a,b\mid r,r^2\rangle\), where \(r=[a,b]\). The boundary of \(F(g,r^2)\) traverses a square twice, producing four parallel dual edges to \(F(g,r)\), each with weight \(2\) and sign \(-1\).

Figure~\ref{fig:dual-local} illustrates Definition~\ref{def:dual-graph} for three faces \(F_1,F_2,F_3\) incident with one edge \(e\). They become the vertices of a dual triangle, whose edges are \(\alpha_{ij}=\alpha(e;F_i,F_j)\). The weights and signs are given by the formulas above: incidences of the same sign give a negative dual edge, while incidences of opposite signs give a positive one.

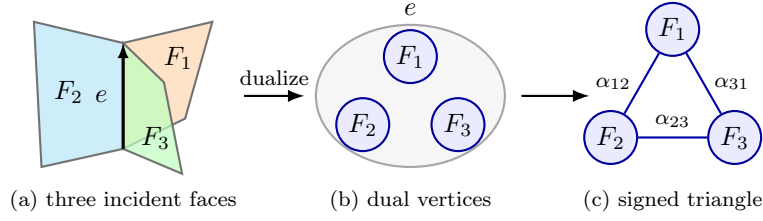
\begin{figure}[htbp]
\centering
\begin{tikzpicture}[
  >=Latex,
  every node/.style={font=\small},
  original face/.style={draw=black!55,thick,fill opacity=.72},
  dual vertex/.style={
    circle,draw=blue!60!black,fill=blue!8,thick,
    minimum size=7mm,inner sep=1pt
  },
  dual edge/.style={draw=blue!65!black,thick}
]
  \begin{scope}[shift={(-4.55,0)}]
    \coordinate (p) at (0,-0.70);
    \coordinate (q) at (0,0.70);
    \filldraw[original face,fill=cyan!24]
      (p)--(q)--(-1.18,0.98)--(-1.08,-0.94)--cycle;
    \filldraw[original face,fill=orange!30]
      (p)--(q)--(1.18,0.98)--(0.82,-0.30)--cycle;
    \filldraw[original face,fill=green!25]
      (p)--(q)--(0.54,0.18)--(0.78,-1.03)--cycle;
    \node at (-0.72,0.05) {\(F_2\)};
    \node at (0.73,0.50) {\(F_1\)};
    \node at (0.43,-0.55) {\(F_3\)};
    \draw[black,very thick,-{Latex[length=2mm]}] (p)--(q);
    \node[left=2pt] at (0,0) {\(e\)};
    \node[font=\scriptsize] at (0,-1.38)
      {\textup{(a)} three incident faces};
  \end{scope}

  \draw[thick,-{Latex[length=2.2mm]}] (-2.95,0) --
    node[above,font=\scriptsize] {dualize} (-2.15,0);

  \begin{scope}[shift={(-0.75,0)}]
    \draw[draw=gray!70,fill=gray!7,thick]
      (0,0) ellipse (1.25 and 0.94);
    \node[dual vertex] (h1) at (0,0.52) {\(F_1\)};
    \node[dual vertex] (h2) at (-0.63,-0.38) {\(F_2\)};
    \node[dual vertex] (h3) at (0.63,-0.38) {\(F_3\)};
    \node[above] at (0,0.94) {\(e\)};
    \node[font=\scriptsize] at (0,-1.38)
      {\textup{(b)} dual vertices};
  \end{scope}

  \draw[thick,-{Latex[length=2.2mm]}] (0.72,0) -- (1.62,0);

  \begin{scope}[shift={(2.72,0)}]
    \node[dual vertex] (a1) at (0,0.88) {\(F_1\)};
    \node[dual vertex] (a2) at (-0.82,-0.55) {\(F_2\)};
    \node[dual vertex] (a3) at (0.82,-0.55) {\(F_3\)};
    \draw[dual edge] (a1)--node[left,font=\scriptsize]
      {\(\alpha_{12}\)} (a2);
    \draw[dual edge] (a2)--node[above,font=\scriptsize]
      {\(\alpha_{23}\)} (a3);
    \draw[dual edge] (a3)--node[right,font=\scriptsize]
      {\(\alpha_{31}\)} (a1);
    \node[font=\scriptsize] at (0,-1.38)
      {\textup{(c)} signed triangle};
  \end{scope}
\end{tikzpicture}
\caption{The dual triangle produced by an edge incident with three faces.}
\label{fig:dual-local}
\end{figure}

We use counting measure on dual vertices, so \(\|f\|_2^2=\sum_F|f(F)|^2\). The weighted degree is
\begin{equation}
\label{eq:dual-degree}
  d^\vee(F)
  :=\sum_{\alpha\ni F}w(\alpha)
  =\sum_e\sum_{\substack{G\ne F}}
  |\eta(e,F)\eta(e,G)|.
\end{equation}
Section~\ref{sec:cayley} gives uniform bounds on \(|\eta(e,F)|\), on the number of edges incident with each face, and on the number of faces incident with each edge. Hence the number and size of the nonzero terms in \eqref{eq:dual-degree} are uniformly bounded, so \(d^\vee\) is bounded.

For a finitely supported function \(f\) on \(V(X^\vee)=\cF(X)\), define the unnormalized weighted signed Laplacian by
\[
  (L_\sigma^\vee f)(F)
  :=\sum_{\alpha=\{F,G\}}w(\alpha)
  \bigl(f(F)-\sigma(\alpha)f(G)\bigr),
\]
where the sum runs over dual edges incident with \(F\), \(G\) denotes the other endpoint, and parallel edges are counted separately. By the degree bound, \(L_\sigma^\vee\) extends uniquely to a bounded self-adjoint operator on \(\ell^2(V(X^\vee))\), with quadratic form
\begin{equation}
\label{eq:signed-energy}
  \langle L_\sigma^\vee f,f\rangle
  =\sum_{\alpha=\{F,G\}}
  w(\alpha)|f(F)-\sigma(\alpha)f(G)|^2.
\end{equation}
Each unoriented dual edge occurs once in this sum, with its weight; parallel edges are counted separately. A positive dual edge contributes \(w(\alpha)|f(F)-f(G)|^2\), and a negative dual edge contributes \(w(\alpha)|f(F)+f(G)|^2\).

A weighted signed graph is \(\Gamma\)-periodic if \(\Gamma\) acts freely on vertices by automorphisms preserving weights and signs. Let \(D\) be a set of vertex-orbit representatives. The graph has \emph{finite range} if there is a finite \(K\subset\Gamma\) such that for all \(u,v\in D\) and \(k\in\Gamma\),
\[
 u\sim kv\quad\Longrightarrow\quad k\in K.
\]
Here \(\sim\) denotes adjacency, and \(v\) may equal \(u\). For a discussion of group-invariant finite-range random walks on graphs with finite quotient, see \cite{BekkaHarpeValette2008}.

\Needspace{3\baselineskip}
\begin{theorem}
\label{thm:periodic-dual}
\label{thm:dual-formula}
The graph \(X^\vee\) is \(\Gamma\)-periodic and finite-range, with one vertex layer for each relator. For one relator it is a signed Cayley multigraph. With \(V(F)=q(F)-d^\vee(F)\),
\[
 \Delta_2=L_\sigma^\vee+V.
\]
Consequently,
\[
 \|\partial_2f\|_2^2
 =\sum_{\alpha=\{F,G\}}w(\alpha)|f(F)-\sigma(\alpha)f(G)|^2
  +\sum_FV(F)|f(F)|^2.
\]
\end{theorem}

\begin{proof}
By \eqref{eq:face-parametrization}, the relation layers \(\Gamma\times\{r\}\), \(r\in R\), are the free \(\Gamma\)-orbits of \(X^\vee\). Dual edges may join vertices within or between layers. By \eqref{eq:boundary-equivariance}, left translation preserves dual edges, weights and signs, so \(X^\vee\) is \(\Gamma\)-periodic and \(q,d^\vee,V\) are \(\Gamma\)-invariant.

The finite presentation and local finiteness of \(X\) make \(X^\vee\) locally finite by construction. The finite representative set \(D=\{F(1,r):r\in R\}\) therefore has finitely many neighbors, each uniquely of the form \(kv\) with \(v\in D\) and \(k\in\Gamma\). These \(k\) form a finite set \(K\), proving finite range. With one relator, the action is transitive, giving a signed Cayley multigraph. The graph need not be connected.

For \(F\ne G\), the off-diagonal coefficients satisfy
\[
 \begin{aligned}
 \langle\delta_F,(L_\sigma^\vee+V)\delta_G\rangle
 &=-\sum_{\alpha=\{F,G\}}w(\alpha)\sigma(\alpha)\\
 &=\sum_e\eta(e,F)\eta(e,G)
 =\langle\delta_F,\Delta_2\delta_G\rangle.
\end{aligned}
\]
The diagonal coefficients agree by definition:
\[
 d^\vee(F)+V(F)=q(F)=\langle\delta_F,\Delta_2\delta_F\rangle.
\]
Since both operators are bounded, \(\Delta_2=L_\sigma^\vee+V\). The quadratic-form identity follows from \eqref{eq:signed-energy}.
\end{proof}

Changing face orientations conjugates \(L_\sigma^\vee+V\) by a diagonal operator with entries \(\pm1\); changing edge orientations leaves it unchanged. Its spectrum is therefore independent of the chosen cell orientations.

For surface groups, the dual formula turns the face spectral problem into an ordinary graph spectral problem on the dual graph.

\begin{corollary}
\label{cor:surface-gaps}
Let \(\Gamma_g\) be the fundamental group of a closed orientable surface of genus \(g\ge1\), with presentation
\[
 \Gamma_g=\left\langle a_1,b_1,\ldots,a_g,b_g
       \ \middle|\ \prod_{j=1}^g[a_j,b_j]\right\rangle,
\]
where \([a,b]=aba^{-1}b^{-1}\). The associated Cayley \(2\)-complex \(X_g\) has \(\lambda_0(X_g)=\lambda_2^+(X_g)=0\) in the torus case \(g=1\), and both gaps are positive for \(g\ge2\).
\end{corollary}

\begin{proof}
The universal cover is the usual polygonal tessellation \cite{LyndonSchupp1977}. Neighboring faces traverse each common edge in opposite directions, so
\[
 \eta(e,F)\eta(e,G)=-1,\qquad w(\alpha)=\sigma(\alpha)=1.
\]
Each face has \(4g\) boundary edges and \(4g\) dual edges. Theorem~\ref{thm:dual-formula} therefore gives
\[
 q(F)=d^\vee(F)=4g,\qquad V=0,\qquad \Delta_2=L_\sigma^\vee.
\]
Thus \(\Delta_2\) is the ordinary Laplacian of the connected dual Cayley multigraph of \(\Gamma_g\). Its \(\ell^2\) kernel is zero, since the graph is infinite and connected. Kesten's theorem yields
\[
 \begin{aligned}
 \lambda_0(X_g)>0
 &\quad\Longleftrightarrow\quad\Gamma_g\text{ is nonamenable}\\
 &\quad\Longleftrightarrow\quad\lambda_2^+(X_g)>0.
 \end{aligned}
\]
Now \(\Gamma_1\cong\Z^2\) is amenable. For \(g\ge2\), the epimorphism
\[
 \Gamma_g\twoheadrightarrow F_g,\qquad
 a_j\mapsto a_j,\quad b_j\mapsto1,
\]
shows that \(\Gamma_g\) is nonamenable.
\end{proof}

Dubashinskiy obtains the positive coexact gap for \(g\ge2\) in \cite[Proposition~5.4]{Dubashinskiy2025}. By Theorem~\ref{thm:spectral-transfer} below, this agrees with Corollary~\ref{cor:surface-gaps}.

\section{Spectral transfer and closed boundaries}
\label{sec:spectral-transfer}

To relate the face operator to the Hodge Laplacian on edges, decompose edge chains into limits of gradients, harmonic chains, and limits of face boundaries. We use the boundedness of \(\partial_1\) and \(\partial_2\), which follows from the finiteness of the presentation. Let \(\cH=\ker\partial_1\cap\ker\partial_2^*\) be the harmonic edge space. We have the weak Hodge decomposition \cite{Eckmann1945,BruningLesch1992}:
\begin{equation}
\label{eq:hodge-decomposition}
 C_1^{(2)}(X)=\overline{\ran\partial_1^*}\oplus\cH
             \oplus\overline{\ran\partial_2}.
\end{equation}

\begin{theorem}
\label{thm:spectral-transfer}
The summands in \eqref{eq:hodge-decomposition} reduce \(\Delta_1\), which vanishes on \(\cH\). On the other two summands,
\begin{align}
 \Delta_1|_{\overline{\ran\partial_1^*}}
 &\simeq\Delta_0|_{(\ker\partial_1^*)^\perp},
 \label{eq:exact-transfer}\\
 \Delta_1|_{\overline{\ran\partial_2}}
 &\simeq\Delta_2|_{(\ker\partial_2)^\perp}.
 \label{eq:coexact-transfer}
\end{align}
Here \(\simeq\) means orthogonal equivalence. In particular,
\begin{equation}
\label{eq:positive-spectrum}
 \spec_+(\Delta_1)=\spec_+(\Delta_0)\cup\spec_+(L_\sigma^\vee+V).
\end{equation}
\end{theorem}

\begin{proof}
We briefly recall the proof of the weak Hodge decomposition first. The identity \(\partial_1\partial_2=0\) gives
\[
 \langle\partial_1^*u,\partial_2v\rangle
 =\langle u,\partial_1\partial_2v\rangle=0.
\]
Using \((\ran T)^\perp=\ker T^*\), we obtain
\[
 \bigl(\overline{\ran\partial_1^*}\oplus
       \overline{\ran\partial_2}\bigr)^\perp
 =\ker\partial_1\cap\ker\partial_2^*=\cH,
\]
which proves \eqref{eq:hodge-decomposition}. Relative to this decomposition,
\[
 \Delta_1=
 (\partial_1^*\partial_1)|_{\overline{\ran\partial_1^*}}
 \oplus0_{\cH}\oplus
 (\partial_2\partial_2^*)|_{\overline{\ran\partial_2}}.
\]

For a bounded operator \(T\) between real Hilbert spaces, its polar decomposition \(T=U(T^*T)^{1/2}\) \cite[Theorem~VIII.3.11]{Conway1990} gives a partial isometry \(U\) whose restriction is an orthogonal isomorphism:
\[
 U:(\ker T)^\perp\longrightarrow\overline{\ran T},\qquad
 U(T^*T)=(TT^*)U.
\]
Taking \(T=\partial_1^*\) and \(T=\partial_2\) proves \eqref{eq:exact-transfer} and \eqref{eq:coexact-transfer}, respectively. The omitted kernels and \(\cH\) contribute only zero, so \(\Delta_2=L_\sigma^\vee+V\) gives \eqref{eq:positive-spectrum}.
\end{proof}

By \eqref{eq:coexact-transfer}, the coexact gap is the spectral bottom of \(\Delta_2|_{(\ker\partial_2)^\perp}\), while the face gap omits zero. This restriction has no zero eigenvalue, so zero, if in its spectrum, is approached by positive spectrum. Removing zero leaves the infimum unchanged, so the dual formula gives
\begin{equation}
\label{eq:bottom-identity}
 \lambda_{1,\mathrm{coex}}(X)=\lambda_2^+(X)
 =\inf\spec_+(L_\sigma^\vee+V).
\end{equation}
We use the convention \(\inf\varnothing=+\infty\).

\begin{proof}[Proof of Theorem~\ref{thm:main}]
The dual construction and operator identity are established in Theorem~\ref{thm:dual-formula}. The spectral statement follows from Theorem~\ref{thm:spectral-transfer}, and the gap identity is~\eqref{eq:bottom-identity}.
\end{proof}

For the bounded boundary operator \(\partial_2\), the following three conditions are equivalent:
\begin{equation}
\label{eq:closed-range}
\begin{aligned}
 \lambda_2^+(X)>0
 &\quad\Longleftrightarrow\quad
 \exists C>0:\ \|f\|_2\le C\|\partial_2f\|_2
 \quad(\forall f\perp\ker\partial_2)\\
 &\quad\Longleftrightarrow\quad \ran\partial_2\text{ is closed}.
\end{aligned}
\end{equation}
The first equivalence follows from \(\langle\Delta_2f,f\rangle=\|\partial_2f\|_2^2\) and the spectral theorem \cite[Theorem~IX.2.2]{Conway1990}. The second is the closed-range theorem \cite[Proposition~XI.6.1]{Conway1990}.

Let \(C_c(\cF(X))\) denote the finitely supported face chains. If \(\partial_2\) is injective, all these chains satisfy \(f\perp\ker\partial_2\). Boundedness and density then give the formula used to prove gap vanishing in Section~\ref{sec:bs}:
\[
 \lambda_2^+(X)=\inf_{0\ne f\in C_c(\cF(X))}
                       \frac{\|\partial_2f\|_2^2}{\|f\|_2^2}.
\]
If the kernel is nontrivial, approximations to a nonzero kernel vector already make this unrestricted infimum zero. Small Rayleigh quotients alone therefore do not prove vanishing of the face gap.

\begin{proposition}
\label{prop:presentation-independence}
For two finite presentations \(\cP,\cQ\) of the same group,
\[
 \lambda_2^+(X_{\cP})>0\quad\Longleftrightarrow\quad
 \lambda_2^+(X_{\cQ})>0.
\]
By \eqref{eq:bottom-identity}, the same equivalence holds for the coexact gap.
\end{proposition}

\begin{proof}
Tietze's theorem \cite[Proposition~II.2.1]{LyndonSchupp1977} connects any two finite presentations by adding or removing a redundant relator or a generator with its defining relator. By \eqref{eq:closed-range}, it suffices to show that the two additions preserve closedness of the boundary range.

First add a relator \(u\) implied by the old relators, obtaining \(\widetilde{\cP}=\langle S\mid R,u\rangle\) of the same group \(\Gamma\). The added faces \(F(g,u)\) are indexed by \(g\in\Gamma\), giving
\[
 C_2^{(2)}(X_{\widetilde{\cP}})
 \cong C_2^{(2)}(X_{\cP})\oplus\ell^2(\Gamma).
\]
The two summands represent the old and added face chains respectively. The van Kampen lemma \cite[Proposition~III.9.2]{LyndonSchupp1977} gives a finitely supported old face chain \(c\) whose boundary equals that of \(F(1,u)\). Define \(B:\ell^2(\Gamma)\to C_2^{(2)}(X_{\cP})\) by \(B(F(g,u))=g\cdot c\). Each coefficient of \(c\) contributes a multiple of an isometric translation, so \(\|B\|\le\sum_F|c(F)|\). Since cellular boundaries commute with translations, each added face and its image under \(B\) have the same boundary. Thus, for \(f\in C_2^{(2)}(X_{\cP})\) and \(h\in\ell^2(\Gamma)\), the new boundary operator satisfies
\[
 \widetilde\partial_2(f,h)
 =\partial_2^{\cP}f+\partial_2^{\cP}Bh
 =\partial_2^{\cP}(f+Bh).
\]
Taking \(h=0\) realizes every old boundary, so
\[
 \ran\widetilde\partial_2=\ran\partial_2^{\cP}.
\]

For the generator addition, take \(\widetilde{\cP}=\langle S\cup\{y\}\mid R,yw^{-1}\rangle\), where \(y\) is a new generator and \(w\) is a word in the old generators. The group remains \(\Gamma\). The added edges \(e(g,y):g\to gw\) and added faces \(F(g,yw^{-1})\) are indexed by \(g\in\Gamma\), giving
\[
 C_k^{(2)}(X_{\widetilde{\cP}})
 \cong C_k^{(2)}(X_{\cP})\oplus\ell^2(\Gamma),
 \qquad k=1,2.
\]
The extra summand records the added \(y\)-edges for \(k=1\) and the added faces for \(k=2\). Define the linear operator \(B_w:\ell^2(\Gamma)\to C_1^{(2)}(X_{\cP})\) by sending \(e(g,y)\) to the oriented edge chain of the old path labeled \(w\) from \(g\) to \(gw\). It is bounded by the finite translation argument above.

The boundary of an added face follows its \(y\)-edge and returns along this old path. Hence, for \(f\in C_2^{(2)}(X_{\cP})\) and \(h\in\ell^2(\Gamma)\), the boundary in the old and new edge coordinates is
\[
 \widetilde\partial_2(f,h)=(\partial_2^{\cP}f-B_wh,h).
\]
For \(v\in C_1^{(2)}(X_{\cP})\) and \(h\in\ell^2(\Gamma)\), define the bounded invertible change of edge coordinates
\[
 S(v,h)=(v+B_wh,h).
\]
Adding \(B_wh\) to the old edge coordinate cancels the contribution \(-B_wh\), giving
\[
 \begin{aligned}
 S\widetilde\partial_2(f,h)&=(\partial_2^{\cP}f,h),\\
 S(\ran\widetilde\partial_2)&=\ran\partial_2^{\cP}\oplus\ell^2(\Gamma).
 \end{aligned}
\]
Since \(S\) and \(S^{-1}\) are both bounded linear operators, \(S\) is a homeomorphism and preserves closedness in both directions. Convergence in the Hilbert direct sum is componentwise, so \(\ran\partial_2^{\cP}\oplus\ell^2(\Gamma)\) is closed if and only if \(\ran\partial_2^{\cP}\) is closed. The displayed identity therefore gives the required equivalence. Both additions, and hence their inverses, preserve closedness.
\end{proof}

The numerical gap values may depend on the presentation; see Proposition~\ref{ex:cyclic-gap} and the discussion following it.

\section{The Baumslag--Solitar groups}
\label{sec:bs}

We study the face spectral gap of Baumslag--Solitar groups using test functions on the signed dual graph. Let
\[
  \Gamma=\operatorname{BS}(1,m)
  =\langle a,b\mid aba^{-1}b^{-m}\rangle,
  \qquad m\ge2.
\]
Let \(X_m\) be the Cayley \(2\)-complex of this presentation. Put \(r=aba^{-1}b^{-m}\) and \(F(g)=F(g,r)\); the faces form one \(\Gamma\)-orbit by \eqref{eq:face-parametrization}.

For illustration, take \(m=2\). Put \(H=\langle b\rangle\) and call \(gH\) a \(b\)-line. The identity \(b^2a=ab\) gives
\[
 gb^{2j}a=gab^j,\qquad gb^{2j+1}a=gbab^j.
\]
The \(a\)-edges from \(gb^{2j}\) and \(gb^{2j+1}\) reach the successor lines \(gaH\) and \(gbaH\), respectively; the predecessor is \(ga^{-1}H\). Thus the adjacency graph of the \(b\)-lines is an unrooted binary tree, called the Bass--Serre tree \cite[I.5]{Serre1980}, with one predecessor and two successors at each vertex. In Figure~\ref{fig:bs-binary}, the relation face has one \(b\)-edge on the right and two on the left; the dashed \(a\)-edge joins \(gb\) to \(gba\) on the \(b\)-line \(gbaH\).

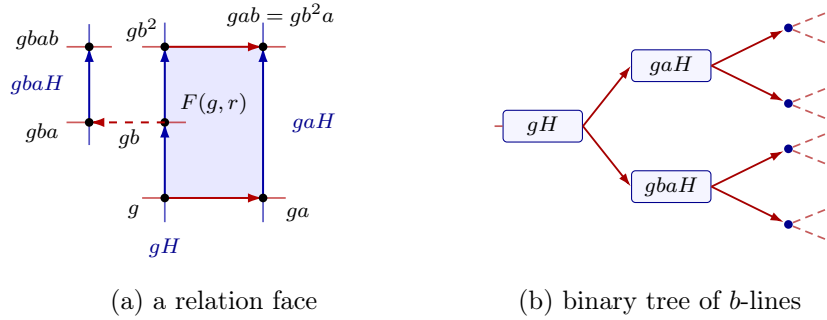
\begin{figure}[htbp]
\centering
\begin{tikzpicture}[
  x=1cm,
  y=1cm,
  every node/.style={font=\scriptsize},
  cayley vertex/.style={circle,fill=black,inner sep=1.15pt},
  b edge/.style={draw=blue!65!black,line width=.75pt,
    -{Latex[length=2mm,width=1.25mm]}},
  a edge/.style={draw=red!70!black,line width=.75pt,
    -{Latex[length=2mm,width=1.25mm]}},
  b germ/.style={draw=blue!65!black,line width=.65pt,opacity=.65},
  a germ/.style={draw=red!70!black,line width=.65pt,opacity=.65},
  line vertex/.style={rounded corners=1.5pt,draw=blue!55!black,
    fill=blue!4,minimum width=1.05cm,minimum height=.42cm,inner sep=1pt},
  tree edge/.style={draw=red!65!black,line width=.7pt,
    -{Latex[length=1.8mm,width=1.15mm]}},
  tree germ/.style={draw=red!65!black,line width=.65pt,
    densely dashed,opacity=.65}
]
  % Reflect the geometry across y=x; typeset the labels separately.
  \begin{scope}[cm={0,1,1,0,(0,0)}]
  % One dyadic relation cell and the second outgoing sheet.
  \coordinate (u0) at (0,1.65);
  \coordinate (u1) at (2,1.65);
  \coordinate (v0) at (0,.35);
  \coordinate (v1) at (1,.35);
  \coordinate (v2) at (2,.35);
  \filldraw[fill=blue!9,draw=black!35,line width=.45pt]
    (v0)--(u0)--(u1)--(v2)--(v1)--cycle;

  \draw[b edge] (v0)--(v1);
  \draw[b edge] (v1)--(v2);
  \draw[b edge] (u0)--(u1);
  \draw[a edge] (v0)--(u0);
  \draw[a edge] (v2)--(u1);

  \coordinate (w0) at (1,-.65);
  \coordinate (w1) at (2,-.65);

  % Omitted b-line continuations and a-edge germs.
  \draw[b germ] (v0)--++(-.32,0);
  \draw[b germ] (v2)--++(.32,0);
  \draw[b germ] (u0)--++(-.32,0);
  \draw[b germ] (u1)--++(.32,0);
  \draw[b germ] (w0)--++(-.30,0);
  \draw[b germ] (w1)--++(.30,0);
  \draw[a germ] (v0)--++(0,-.30);
  \draw[a germ] (v1)--++(0,.28);
  \draw[a germ] (v2)--++(0,-.30);
  \draw[a germ] (u0)--++(0,.30);
  \draw[a germ] (u1)--++(0,.30);
  \draw[a germ] (w0)--++(0,-.30);
  \draw[a germ] (w1)--++(0,.30);
  \draw[a germ] (w1)--++(0,-.30);

  \node[cayley vertex] at (v0) {};
  \node[cayley vertex] at (v1) {};
  \node[cayley vertex] at (v2) {};
  \node[cayley vertex] at (u0) {};
  \node[cayley vertex] at (u1) {};

  \draw[a edge,dashed] (v1)--(w0);
  \draw[b edge] (w0)--(w1);
  \node[cayley vertex] at (w0) {};
  \node[cayley vertex] at (w1) {};
  \end{scope}

  % Upright labels for the reflected relation face.
  \node at (1,1.25) {\(F(g,r)\)};
  \node[anchor=east] at (.15,-.18) {\(g\)};
  \node[anchor=east] at (.15,.78) {\(gb\)};
  \node at (.05,2.22) {\(gb^2\)};
  \node[anchor=west] at (1.82,-.22) {\(ga\)};
  \node[anchor=south] at (1.88,2.14) {\(gab=gb^2a\)};
  \node[blue!55!black,below] at (.35,-.40) {\(gH\)};
  \node[blue!55!black,anchor=west] at (1.90,1) {\(gaH\)};
  \node[anchor=east] at (-.92,.87) {\(gba\)};
  \node[anchor=east] at (-.92,2.10) {\(gbab\)};
  \node[blue!55!black,anchor=east] at (-.85,1.50) {\(gbaH\)};
  \node[font=\small] at (1,-1.38) {\textup{(a) a relation face}};

  % The binary tree rotated clockwise.
  \begin{scope}[xshift=5.25cm]
    \node[line vertex] (root) at (.10,.95) {\(gH\)};
    \node[line vertex] (upper) at (1.80,1.75) {\(gaH\)};
    \node[line vertex] (lower) at (1.80,.15) {\(gbaH\)};
    \node[circle,fill=blue!55!black,inner sep=1.15pt] (ll) at (3.35,2.25) {};
    \node[circle,fill=blue!55!black,inner sep=1.15pt] (lr) at (3.35,1.25) {};
    \node[circle,fill=blue!55!black,inner sep=1.15pt] (rl) at (3.35,.65) {};
    \node[circle,fill=blue!55!black,inner sep=1.15pt] (rr) at (3.35,-.35) {};

    \draw[tree edge] (root.east)--(upper.west);
    \draw[tree edge] (root.east)--(lower.west);
    \draw[tree edge] (upper.east)--(ll);
    \draw[tree edge] (upper.east)--(lr);
    \draw[tree edge] (lower.east)--(rl);
    \draw[tree edge] (lower.east)--(rr);
    \foreach \leaf in {ll,lr,rl,rr}{
      \draw[tree germ] (\leaf)--++(.50,.20);
      \draw[tree germ] (\leaf)--++(.50,-.20);
    }
    \draw[tree germ] (root)--++(-.65,0);

    \node[font=\small] at (1.65,-1.38)
      {\textup{(b) binary tree of \(b\)-lines}};
  \end{scope}
\end{tikzpicture}
\caption{A relation face and the binary tree of \(b\)-lines in \(\operatorname{BS}(1,2)\).}
\label{fig:bs-binary}
\end{figure}

A square-summable face chain on \(X_m\) is specified by a function \(f:\Gamma\to\R\), where \(f(g)\) is the coefficient of the face \(F(g)\) based at \(g\in\Gamma\). Thus
\[
 C_2^{(2)}(X_m)\cong\ell^2(\Gamma),\qquad
 \sum_{g\in\Gamma}f(g)F(g)\longleftrightarrow f.
\]
\Needspace{3\baselineskip}
\begin{lemma}
\label{lem:bs-boundary}
All edges of \(X_m^\vee\) have weight \(1\). Each \(e(g,a)\) gives a positive dual edge joining \(F(g)\) and \(F(gb^{-m})\), while each \(e(g,b)\) gives a clique on
\[
 \{F(ga^{-1}),F(g),F(gb^{-1}),\ldots,F(gb^{-(m-1)})\},
\]
with positive edges from \(F(ga^{-1})\) to the other vertices and negative edges between those vertices. The potential is constant, \(V=1-m^2\), and
\begin{equation}
\label{eq:bs-dual}
 \Delta_2=L_\sigma^\vee+(1-m^2)I.
\end{equation}
Moreover, \(\partial_2\) is injective.
\end{lemma}

\begin{proof}
We use the affine model (cf.~\cite[III.5]{deLaHarpe2000})
\[
  \Gamma=\Z\ltimes\Z[1/m],\qquad
  (x,y)(x',y')=(x+x',y+m^x y'),
\]
where \(\Z[1/m]=\{k/m^n:k\in\Z,\ n\ge0\}\), \(a=(1,0)\), and \(b=(0,1)\). The proper prefixes of \(r\) visit the distinct vertices
\[
 (0,0),(1,0),(1,m),(0,m),\ldots,(0,1).
\]
Thus the lifted face boundary is embedded. Its positively oriented boundary reads \(r=aba^{-1}b^{-m}\): the \(a,b\) steps follow the generator edge orientations and contribute \(+1\), while the \(a^{-1},b^{-1}\) steps run against them and contribute \(-1\). For \(g\in\Gamma\) and \(0\le j<m\), the nonzero incidence numbers are
\begin{align*}
 \eta(e(g,a),F(g))&=1,&
 \eta(e(g,a),F(gb^{-m}))&=-1,\\
 \eta(e(g,b),F(ga^{-1}))&=1,&
 \eta(e(g,b),F(gb^{-j}))&=-1.
\end{align*}
These incidences give the stated signed edges, all of weight \(1\). For \(1\le k<m\), the faces \(F(g)\) and \(F(gb^{-k})\) occur together in the cliques associated with \(e(gb^j,b)\), \(0\le j<m-k\), giving \(m-k\) negative dual edges between them.

For \(m=2\), the three \(b\)-edge cliques containing \(F(g)\) have vertex sets
\[
  \begin{gathered}
  \{F(ga^{-1}),F(g),F(gb^{-1})\},\qquad
  \{F(gba^{-1}),F(gb),F(g)\},\\
  \{F(g),F(ga),F(gab^{-1})\}.
  \end{gathered}
\]
Each triangle has one negative and two positive edges. The \(a\)-edges add positive edges to \(F(gb^{\pm2})\), giving eight distinct neighbors and \(d^\vee(F(g))=8\), as in Figure~\ref{fig:bs-dual-neighborhood}.

\begin{figure}[htbp]
\centering
\begin{tikzpicture}[
  x=1cm,
  y=1cm,
  every node/.style={font=\scriptsize},
  dual vertex/.style={
    rounded corners=1.2pt,draw=blue!60!black,fill=blue!6,
    line width=.65pt,minimum height=.46cm,inner sep=2.2pt
  },
  center vertex/.style={dual vertex,fill=blue!16,line width=.9pt},
  clique/.style={fill=gray!7,draw=none},
  positive edge/.style={draw=blue!65!black,line width=.85pt},
  positive bypass/.style={
    draw=blue!65!black,line width=.85pt,
    preaction={draw=white,line width=2.5pt}
  },
  negative edge/.style={draw=red!70!black,dashed,line width=.85pt},
  clique label/.style={font=\scriptsize,text=black!48}
]
  % Reflect the graph geometry across y=x; keep labels upright.
  \begin{scope}[cm={0,1,1,0,(0,0)}]
  \coordinate (c) at (0,0);
  \coordinate (bm2) at (-3.50,0);
  \coordinate (bm1) at (-1.75,0);
  \coordinate (bp1) at (1.75,0);
  \coordinate (bp2) at (3.50,0);
  \coordinate (aminus) at (0,-1.65);
  \coordinate (baminus) at (1.75,-1.65);
  \coordinate (aplus) at (0,1.65);
  \coordinate (abminus) at (-3.50,1.65);

  % The three complete graphs contributed by b-edges.
  \path[clique] (c)--(bm1)--(aminus)--cycle;
  \path[clique] (c)--(bp1)--(baminus)--cycle;
  \path[clique] (c)--(aplus)--(abminus)--cycle;

  % Positive dual edges.
  \draw[positive edge] (c)--(aminus);
  \draw[positive edge] (aminus)--(bm1);
  \draw[positive edge] (c)--(baminus);
  \draw[positive edge] (baminus)--(bp1);
  \draw[positive edge] (c)--(aplus);
  \draw[positive edge] (c)--(abminus);

  % Negative dual edges.
  \draw[negative edge] (c)--(bm1);
  \draw[negative edge] (c)--(bp1);
  \draw[negative edge] (aplus)--(abminus);
  \end{scope}

  \node[clique label] at (-.64,-.58) {\(K_3\)};
  \node[clique label] at (-.64,1.17) {\(K_3\)};
  \node[clique label] at (1.10,-1.17) {\(K_3\)};

  % Route long positive edges around the upright vertex boxes.
  \draw[positive bypass] (c)
    .. controls (-2.30,-1.00) and (-2.30,-2.50) .. (bm2);
  \draw[positive bypass] (c)
    .. controls (2.30,1.00) and (2.30,2.50) .. (bp2);

  % Dual vertices are drawn last so that all edges end cleanly.
  \node[center vertex] at (c) {\(F(g)\)};
  \node[dual vertex] at (bm2) {\(F(gb^{-2})\)};
  \node[dual vertex] at (bm1) {\(F(gb^{-1})\)};
  \node[dual vertex] at (bp1) {\(F(gb)\)};
  \node[dual vertex] at (bp2) {\(F(gb^2)\)};
  \node[dual vertex] at (aminus) {\(F(ga^{-1})\)};
  \node[dual vertex] at (baminus) {\(F(gba^{-1})\)};
  \node[dual vertex] at (aplus) {\(F(ga)\)};
  \node[dual vertex] at (abminus) {\(F(gab^{-1})\)};

  % Sign and weight legend.
  \draw[positive edge] (-2.15,-4.25)--(-1.55,-4.25);
  \node[anchor=west] at (-1.47,-4.25) {positive};
  \draw[negative edge] (.20,-4.25)--(.80,-4.25);
  \node[anchor=west] at (.88,-4.25) {negative};
\end{tikzpicture}
\caption{The signed dual neighborhood of \(F(g)\) in \(\operatorname{BS}(1,2)\).}
\label{fig:bs-dual-neighborhood}
\end{figure}
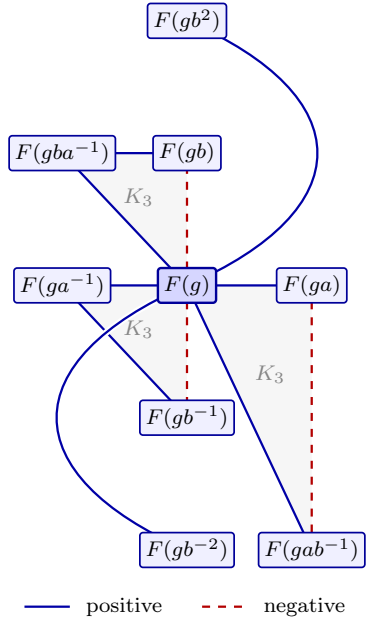

Each face has two \(a\)-edges and \(m+1\) \(b\)-edges with incidence numbers \(\pm1\), so each contributes \(1\) to \(q(F)\) as defined in \eqref{eq:q-def}. Each \(a\)-edge of \(F\) is shared with one other face, and each \(b\)-edge with \(m\) other faces. Each such connection gives a dual edge of weight \(1\), counted in \(d^\vee(F)\) as defined in \eqref{eq:dual-degree}; connections to the same face through different edges of \(X\) are counted separately. Thus
\[
 \begin{gathered}
 q(F)=2+(m+1)=m+3,\\
 d^\vee(F)=2\cdot1+(m+1)\cdot m=2+m(m+1),\\
 V(F)=q(F)-d^\vee(F)=1-m^2,
 \end{gathered}
\]
so \eqref{eq:bs-dual} follows from Theorem~\ref{thm:dual-formula}.

Finally, for \(f\in\ell^2(\Gamma)\), the incidence formulas give \(\langle e(g,a),\partial_2f\rangle=f(g)-f(gb^{-m})\). If \(\partial_2f=0\), then \(f(gb^{km})=f(g)\) for every \(k\in\Z\). Since \(b\) has infinite order,
\[
 \|f\|_2^2\ge\sum_{k\in\Z}|f(gb^{km})|^2
 =\sum_{k\in\Z}|f(g)|^2<\infty
 \quad\Longrightarrow\quad f(g)=0.
\]
Thus \(\ker\partial_2=\{0\}\).
\end{proof}

We identify the dual vertex \(F(g)\) with \(g\in\Gamma\). Let \(f\) be a finitely supported test function. The signed energy of the clique arising from \(e(g,b)\) is
\[
 \begin{aligned}
 &\sum_{j=0}^{m-1}|f(ga^{-1})-f(gb^{-j})|^2
   +\sum_{0\le j<k<m}|f(gb^{-j})+f(gb^{-k})|^2\\
 &\quad=(m-1)\left(|f(ga^{-1})|^2+\sum_{j=0}^{m-1}|f(gb^{-j})|^2\right)\\
 &\qquad+\left|f(ga^{-1})-\sum_{j=0}^{m-1}f(gb^{-j})\right|^2.
 \end{aligned}
\]
Adding the contribution of the dual edge corresponding to \(e(g,a)\) and rearranging the sums, we obtain
\begin{equation}
\label{eq:bs-direct-form}
 \begin{aligned}
 \langle L_\sigma^\vee f,f\rangle
 &=(m^2-1)\|f\|_2^2+\sum_g|f(g)-f(gb^{-m})|^2\\
 &\quad+\sum_g\left|f(ga^{-1})-\sum_{j=0}^{m-1}f(gb^{-j})\right|^2.
 \end{aligned}
\end{equation}
For \(g=(x,y)\) in the affine model and \(j\in\Z\), right multiplication gives
\[
 ga^{-1}=(x-1,y),\qquad gb^{-j}=(x,y-jm^x).
\]

\begin{proof}[Proof of Theorem~\ref{thm:bs-gap}]
By Lemma~\ref{lem:bs-boundary}, \(\partial_2\) is injective. The variational formula in Section~\ref{sec:spectral-transfer} therefore reduces the gap claim to constructing nonzero finitely supported functions with arbitrarily small Rayleigh quotients. The potential \(V=1-m^2\) cancels the norm term in \eqref{eq:bs-direct-form}, so we only need to reduce the last two sums.

We now use complex-valued test functions, with all inner products taken in the Hermitian sense. The first of these sums compares \(f(x,y)\) and \(f(x,y-m^{x+1})\) on the same \(x\)-level. We choose the \(m\)-periodic function \(\chi(y)=e^{2\pi i y/m}\), \(y\in\Z[1/m]\), so that
\[
 \chi(y-m^{x+1})=\chi(y),\qquad\text{for }x\ge0.
\]
For the second sum, the shifted values satisfy
\[
 \sum_{j=0}^{m-1}\chi(y-jm^x)
 =\begin{cases}
   0,&x=0,\\
   m\chi(y),&x\ge1.
  \end{cases}
\]
Thus we use the factor \(m^{-x}\) in front of \(\chi(y)\) to cancel the second remainder for \(x\ge0\).

Set \(E_k=\{0,\ldots,k-1\}\) for \(k\ge1\). For integers \(N,M\ge1\), define
\[
 f=f_{N,M},\qquad
 f_{N,M}(x,y)=m^{-x}\chi(y)\mathbf1_{E_N}(x)\mathbf1_{E_M}(y).
\]
Put \(I=\mathbf1_{E_M}\). For \(t\in\mathbb{N}\),
\begin{equation}
\label{eq:bs-interval-shift}
 \begin{aligned}
 \sum_{y<M}|I(y)-I(y-t)|^2
 &=\sum_{y\ge M}|I(y)-I(y-t)|^2\\
 &=\min\{t,M\}\le t.
 \end{aligned}
\end{equation}
Substituting \(f\) into the last two sums of \eqref{eq:bs-direct-form}, and using \eqref{eq:bs-interval-shift} and Cauchy--Schwarz, gives
\[
 \begin{aligned}
 \sum_g|f(g)-f(gb^{-m})|^2 &=\sum_{x=0}^{N-1}m^{-2x}\sum_{y\in\Z\left[1/m\right]}
      \left|\chi(y)\bigl[I(y)-I(y-m^{x+1})\bigr]\right|^2\\
 &=\sum_{x=0}^{N-1}m^{-2x}
      \sum_{y\in\Z}|I(y)-I(y-m^{x+1})|^2\\
 &\le2m\sum_{x=0}^{N-1}m^{-x}
 \le\frac{2m^2}{m-1},
 \end{aligned}
\]
and
\[
 \begin{aligned}
 &\sum_g\left|f(ga^{-1})-\sum_{j=0}^{m-1}f(gb^{-j})\right|^2\\
 &=Mm^{-2(N-1)}
   +\sum_{x=0}^{N-1}m^{-2x}\sum_{y\in\Z[1/m]}
      \left|\sum_{j=0}^{m-1}\chi(y-jm^x)
      \bigl[I(y)-I(y-jm^x)\bigr]\right|^2\\
 &\le Mm^{-2(N-1)}
   +m\sum_{x=0}^{N-1}m^{-2x}\sum_{j=0}^{m-1}
      \sum_{y\in\Z}|I(y)-I(y-jm^x)|^2\\
 &\le Mm^{-2(N-1)}
   +2m\sum_{x=0}^{N-1}m^{-2x}\sum_{j=0}^{m-1}jm^x\\
 &=Mm^{-2(N-1)}
   +m^2(m-1)\sum_{x=0}^{N-1}m^{-x}\\
 &\le Mm^{-2(N-1)}+m^3.
 \end{aligned}
\]
Hence, with \(C_m=m^3+2m^2/(m-1)\), we obtain by \eqref{eq:bs-dual} and \eqref{eq:bs-direct-form}
\begin{equation}
\label{eq:bs-dual-test-energy}
 \langle\Delta_2f,f\rangle\le Mm^{-2(N-1)}+C_m.
\end{equation}
The total squared norm is
\begin{equation}
\label{eq:bs-direct-norm}
 \|f_{N,M}\|_2^2
 =M\sum_{x=0}^{N-1}m^{-2x}\ge M.
\end{equation}
Hence the variational formula gives
\[
 0\le\lambda_2^+(X_m)
 \le\frac{\langle\Delta_2f_{N,M},f_{N,M}\rangle}{\|f_{N,M}\|_2^2}
 \le m^{-2(N-1)}+\frac{C_m}{M}.
\]
Taking \(M=m^{2N}\) bounds the quotient by \((m^2+C_m)m^{-2N}\to0\). Finally, Theorem~\ref{thm:spectral-transfer} gives \(\lambda_{1,\mathrm{coex}}(X_m)=\lambda_2^+(X_m)=0\).
\end{proof}

\begin{remark}
\label{rem:bs-homology}
In fact, Mattox's Theorem~5.14, applied to \(G_{m,1}\cong\operatorname{BS}(1,m)\), gives \(H_1(\Gamma;\ell^2(\Gamma))\ne0\) \cite{Mattox2016}. In our setting, this means that the unreduced homology group \(\ker\partial_1/\ran\partial_2\) of the \(\ell^2\) chain complex is nonzero. However, since \(\Gamma\) is infinite and amenable, Cheeger and Gromov's theorem implies that the reduced homology group \(\ker\partial_1/\overline{\ran\partial_2}\) is zero \cite[Theorem~0.2]{CheegerGromov1986}. Thus
\[
 \ran\partial_2\subsetneq\ker\partial_1
 =\overline{\ran\partial_2}.
\]
The boundary range is not closed, so \eqref{eq:closed-range} gives \(\lambda_2^+(X_m)=0\), which also answers Dubashinskiy's question. Our proof uses finite face chains without using group homology theory, and gives injectivity and explicit decay of Rayleigh quotients in addition.
\end{remark}

Dubashinskiy's \(2\)-complex is obtained by attaching a face to every embedded cycle of length at most \(D\) in the Cayley graph \cite[Section~1]{Dubashinskiy2025}. In fact, this amounts to adding redundant relators. Proposition~\ref{prop:presentation-independence} then gives the following consequence of Theorem~\ref{thm:bs-gap}.

\begin{corollary}
\label{cor:bs-bounded-cycle}
Let \(D\ge m+3\). The bounded cycle coexact Laplacian of \(\operatorname{BS}(1,m)\) over \(\{a,b\}\), formed from the embedded geometric cycles of length at most \(D\), has zero spectral gap.
\end{corollary}

\begin{proof}
Let \(Y_D\) denote this complex for \(\Gamma=\operatorname{BS}(1,m)\). Since \(D\ge m+3\), it contains all faces of \(X_m\). Each additional cycle reads a word of length at most \(D\) representing \(1\) in \(\Gamma\), so its relation is already implied by \(aba^{-1}b^{-m}\). For fixed \(D\), there are only finitely many such words. Choose a finite filling in \(X_m\) for each word and translate these fillings to the added cycles. The redundant-relator argument in Proposition~\ref{prop:presentation-independence} then gives \(\ran\partial_2^{Y_D}=\ran\partial_2^{X_m}\). By \eqref{eq:closed-range} and Theorem~\ref{thm:bs-gap}, \(\lambda_2^+(Y_D)=0\).

Since \(Y_D\) is simply connected, its coexact space is the closure of the finitely supported graph cycles, and its coexact \(1\)-Laplacian is exactly Dubashinskiy's operator. Theorem~\ref{thm:spectral-transfer} therefore gives the claimed zero gap.
\end{proof}

\section{Amenable groups with zero face gap}
\label{sec:amenable}

In this section, we prove Theorem~\ref{thm:betti-two-gap} using F\o lner cutoffs and infinite cyclic covers. We first introduce F\o lner cutoffs.

\begin{definition}
\label{def:folner-sequence}
Write \(Fg\mathbin{\triangle}F=(Fg\setminus F)\cup(F\setminus Fg)\) for the symmetric difference. A sequence \((F_n)\) of nonempty finite subsets of \(\Gamma\) is a F\o lner sequence if
\[
 \frac{|F_ng\mathbin{\triangle}F_n|}{|F_n|}
 \longrightarrow0
 \qquad\text{for every fixed }g\in\Gamma.
\]
\end{definition}

Such a sequence exists if and only if \(\Gamma\) is amenable, so the quotient criterion above is equivalent to the vertex Rayleigh quotient criterion using boundaries of finite sets \cite{Folner1955,KestenMean1959}. Set \(u_F=|F|^{-1/2}\mathbf1_F\); the normalized indicators \(u_{F_n}\) are the associated F\o lner cutoffs.

We remark that a common convention in the literature on amenability uses \(|gL_n\mathbin{\triangle}L_n|/|L_n|\to0\) \cite[Section~G.5]{BekkaHarpeValette2008}. Inversion identifies the two F\o lner conditions: for \(F_n=L_n^{-1}\), \(|F_ng\mathbin{\triangle}F_n|/|F_n|=|g^{-1}L_n\mathbin{\triangle}L_n|/|L_n|\to0\).

For \(g\in\Gamma\), define \(R_g:\ell^2(\Gamma)\to\ell^2(\Gamma)\) by \((R_gf)(x)=f(xg)\), \(x\in\Gamma\). For the cutoffs \(u_F\), the equivalence \(xg\in F\iff x\in Fg^{-1}\) gives \(R_gu_F=|F|^{-1/2}\mathbf1_{Fg^{-1}}\). Hence
\[
 \|R_gu_F-u_F\|_2^2
 =\frac{|Fg^{-1}\mathbin{\triangle}F|}{|F|}
 =\frac{|Fg\mathbin{\triangle}F|}{|F|}.
\]
The last equality follows by translating the symmetric difference by \(g\).

Let \(S^1=\{z\in\C:|z|=1\}\). We now multiply F\o lner cutoffs by characters to prove the following consequence of Hulanicki's theorem \cite{Hulanicki1964}.

\Needspace{8\baselineskip}
\begin{lemma}
\label{lem:folner-evaluation}
Let \(\Gamma\) be amenable and let \(T=\sum_{g\in E}M_gR_g\colon\ell^2(\Gamma)^k\to\ell^2(\Gamma)^k\) be a bounded linear operator, where \(E\subset\Gamma\) is finite and \(M_g\in M_k(\R)\). If \(T\) is self-adjoint and nonnegative, then
\[
 T(\chi):=\sum_{g\in E}M_g\chi(g)\ge0
\]
for every homomorphism \(\chi:\Gamma\to S^1\).
\end{lemma}

\begin{proof}
By the F\o lner criterion for amenability, choose nonempty finite \(F_n\) with \(|F_ng\mathbin{\triangle}F_n|/|F_n|\to0\) for \(g\in E\), and set
\[
 \xi_n(x)=|F_n|^{-1/2}\chi(x)\mathbf1_{F_n}(x),
 \qquad \|\xi_n\|_2=1.
\]
For \(g\in E\), the identities \(\chi(xg)=\chi(x)\chi(g)\) and \(\mathbf1_{F_n}(xg)=\mathbf1_{F_ng^{-1}}(x)\) give
\[
 \bigl(R_g\xi_n-\chi(g)\xi_n\bigr)(x)
 =\frac{\chi(g)\chi(x)}{\sqrt{|F_n|}}
   \bigl[\mathbf1_{F_ng^{-1}}(x)-\mathbf1_{F_n}(x)\bigr].
\]
Using \(|\chi(g)\chi(x)|=1\), we obtain
\[
 \begin{aligned}
 \|R_g\xi_n-\chi(g)\xi_n\|_2^2
 &=\frac1{|F_n|}\sum_{x\in\Gamma}
   |\mathbf1_{F_ng^{-1}}(x)-\mathbf1_{F_n}(x)|^2\\
 &=\frac1{|F_n|}\sum_{x\in\Gamma}
   \mathbf1_{F_ng^{-1}\mathbin{\triangle}F_n}(x)\\
 &=\frac{|F_ng^{-1}\mathbin{\triangle}F_n|}{|F_n|}
 =\frac{|F_ng\mathbin{\triangle}F_n|}{|F_n|}
   \longrightarrow0.
 \end{aligned}
\]
For \(v\in\C^k\), put \(f_{n,v}(x)=\xi_n(x)v\). Expanding \(T\) and \(T(\chi)\) gives
\[
 (Tf_{n,v})(x)-\xi_n(x)T(\chi)v
 =\sum_{g\in E}
   [\xi_n(xg)-\chi(g)\xi_n(x)]M_gv.
\]
The triangle inequality gives
\[
 \begin{aligned}
 \|Tf_{n,v}-\xi_nT(\chi)v\|_2
 &\le\sum_{g\in E}
   \bigl\|[R_g\xi_n-\chi(g)\xi_n]M_gv\bigr\|_2\\
 &=\sum_{g\in E}\|M_gv\|\,
   \|R_g\xi_n-\chi(g)\xi_n\|_2
   \longrightarrow0,
 \end{aligned}
\]
since \(E\) is finite.
Recall that \(\xi_n\) vanishes outside \(F_n\) and has \(\|\xi_n\|_2=1\). Hence \(\|f_{n,v}\|_2=\|v\|\), and
\[
 \begin{aligned}
 \langle \xi_nT(\chi)v,f_{n,v}\rangle
 &=\sum_{x\in\Gamma}|\xi_n(x)|^2\langle T(\chi)v,v\rangle\\
 &=\|\xi_n\|_2^2\langle T(\chi)v,v\rangle
 =\langle T(\chi)v,v\rangle.
 \end{aligned}
\]
The Cauchy--Schwarz inequality gives
\[
 \begin{aligned}
 \bigl|\langle Tf_{n,v},f_{n,v}\rangle-\langle T(\chi)v,v\rangle\bigr|
 &=\bigl|\langle Tf_{n,v}-\xi_nT(\chi)v,f_{n,v}\rangle\bigr|\\
 &\le\|Tf_{n,v}-\xi_nT(\chi)v\|_2\,\|v\|
 \longrightarrow0.
 \end{aligned}
\]
Since \(T\ge0\), every term \(\langle Tf_{n,v},f_{n,v}\rangle\) is nonnegative. Therefore
\[
 \langle T(\chi)v,v\rangle
 =\lim_n\langle Tf_{n,v},f_{n,v}\rangle\ge0.
\]
Since \(v\in\C^k\) is arbitrary, \(T(\chi)\ge0\).
\end{proof}

To connect character evaluation with homology, we turn to infinite cyclic covers. Recall from Section~\ref{sec:cayley} that \(K_{\cP}=\Gamma\backslash X\) is the presentation complex, where \(\Gamma\) acts on its universal cover \(X\) by left deck transformations \(g\mapsto hg\).

\begin{definition}
\label{def:infinite-cyclic-cover}
Let \(\varphi:\Gamma\twoheadrightarrow\Z\) be an epimorphism and \(N=\ker\varphi\). Using the same left action of \(\Gamma\) on \(X\), form the projections
\[
 X\xrightarrow{q_N}X_\varphi:=N\backslash X
 \xrightarrow{p_\varphi}K_{\cP}=\Gamma\backslash X.
\]
The covering \(p_\varphi\) is called the \emph{infinite cyclic cover associated with \(\varphi\)} \cite[Section~2]{Milnor1968}. The map \(\varphi\) labels its vertices by \(v_n\), \(n\in\Z\); the deck group \(\Gamma/N\cong\Z\) is generated by \(t:v_n\mapsto v_{n+1}\).
\end{definition}
For \(\cP=\langle a,b\mid aba^{-1}b^{-1}\rangle\), take \(\varphi(a)=1\) and \(\varphi(b)=0\), so \(N=\langle b\rangle\). Figure~\ref{fig:cyclic-cover} shows the resulting cover: quotienting the square grid by \(N\) turns each vertical \(b\)-line into a loop, while the \(a\)-edges form an integer line. On vertices, the projections are \(a^ib^j\mapsto v_i\mapsto v\). The relator face of \(K_{\cP}\) lifts under \(p_\varphi\) to a face at each level of \(X_\varphi\). These are exactly the images under \(q_N\) of the square faces of \(X\), with the same boundary word \(aba^{-1}b^{-1}\).

\begin{figure}[htbp]
\centering
\begin{tikzpicture}[
  >=Latex,
  every node/.style={font=\small},
  lifted face/.style={fill=blue!8,draw=none},
  cycle edge/.style={
    draw=blue!55!black,line width=.9pt,
    -{Latex[length=1.7mm,width=1.1mm]},
    shorten <=1.7pt,shorten >=1.7pt},
  continuation/.style={draw=black!48,line width=.55pt,densely dashed},
  map arrow/.style={thick,-{Latex[length=2mm]}},
  marked vertex/.style={circle,fill=black,inner sep=1.25pt}
]
  % The universal cover: a square grid, with b-lines vertical.
  \foreach \x in {-2.4,-.8,.8}
    \foreach \y in {5.4,6.3}
      \path[lifted face] (\x,\y) rectangle ++(1.6,.9);
  \foreach \y in {5.4,6.3,7.2}{
    \foreach \x in {-2.4,-.8,.8}
      \draw[cycle edge] (\x,\y)--++(1.6,0);
    \draw[continuation] (-3.1,\y)--(-2.4,\y);
    \draw[continuation] (2.4,\y)--(3.1,\y);
  }
  \foreach \x in {-2.4,-.8,.8,2.4}{
    \foreach \y in {5.4,6.3}
      \draw[cycle edge] (\x,\y)--++(0,.9);
    \draw[continuation] (\x,5.0)--(\x,5.4);
    \draw[continuation] (\x,7.2)--(\x,7.6);
    \foreach \y in {5.4,6.3,7.2}
      \node[marked vertex] at (\x,\y) {};
  }
  \foreach \x/\lab in {-2.4/a^{-1},-.8/1,.8/a,2.4/a^2}
    \node[below left=2pt,font=\scriptsize] at (\x,6.3) {\(\lab\)};
  \node[below left=2pt,font=\scriptsize] at (-.8,5.4) {\(b^{-1}\)};
  \node[above left=2pt,font=\scriptsize] at (-.8,7.2) {\(b\)};
  \node[above=2pt,font=\scriptsize] at (0,6.3) {\(a\)};
  \node[right=2pt,font=\scriptsize] at (-.8,6.75) {\(b\)};
  \node[anchor=west] at (3.7,6.3) {\(X\)};

  \draw[map arrow] (0,4.92)--node[right=3pt] {\(q_N\)} (0,4.17);

  % The cyclic cover: schematic relator shading and bold a-edges.
  \foreach \x in {-2.4,-.8,.8}
    \draw[cycle edge,line width=1.8pt] (\x,3)--node[below=2pt,font=\scriptsize]
      {\(a\)} ++(1.6,0);
  \draw[continuation] (-3.1,3)--(-2.4,3);
  \draw[continuation] (2.4,3)--(3.1,3);
  \node[text=black!55] at (-3.35,3) {\(\cdots\)};
  \node[text=black!55] at (3.35,3) {\(\cdots\)};
  \foreach \x/\i in {-2.4/-1,-.8/0,.8/1,2.4/2}{
    \path[lifted face] (\x,3)
      .. controls (\x-.64,3.55) and (\x-.48,3.95) .. (\x,3.95)
      .. controls (\x+.48,3.95) and (\x+.64,3.55) .. (\x,3)
      -- cycle;
    \draw[cycle edge] (\x,3)
      .. controls (\x-.64,3.55) and (\x-.48,3.95) .. (\x,3.95)
      .. controls (\x+.48,3.95) and (\x+.64,3.55) .. (\x,3);
    \node[font=\scriptsize] at (\x,3.58) {\(b\)};
    \node[marked vertex] at (\x,3) {};
    \node[below=5pt,font=\scriptsize] at (\x,3) {\(v_{\i}\)};
  }
  \node[anchor=west] at (3.7,3.35) {\(X_\varphi=N\backslash X\)};

  \draw[map arrow] (0,2.25)--node[right=3pt] {\(p_\varphi\)} (0,1.33);

  % The presentation complex: generator loops with schematic relator shading.
  \foreach \sgn/\lab in {-1/b,1/a}{
    \path[lifted face] (0,0)
      .. controls (\sgn*.70,-.15) and (\sgn*1.20,.15) .. (\sgn*1.15,.60)
      .. controls (\sgn*1.10,1.0) and (\sgn*.45,1.0) .. (0,0)
      -- cycle;
    \draw[cycle edge] (0,0)
      .. controls (\sgn*.70,-.15) and (\sgn*1.20,.15) .. (\sgn*1.15,.60)
      .. controls (\sgn*1.10,1.0) and (\sgn*.45,1.0) .. (0,0);
    \node[font=\scriptsize] at (\sgn*.67,.50) {\(\lab\)};
  }
  \node[marked vertex] at (0,0) {};
  \node[below=2pt,font=\scriptsize] at (0,0) {\(v\)};
  \node[anchor=west] at (3.7,.4) {\(K_{\cP}\)};
  \node[font=\scriptsize] at (0,-.60) {\(r=aba^{-1}b^{-1}\)};
\end{tikzpicture}
\caption{The infinite cyclic cover of the torus; shading and thick edges schematically indicate the commutator faces.}
\label{fig:cyclic-cover}
\end{figure}
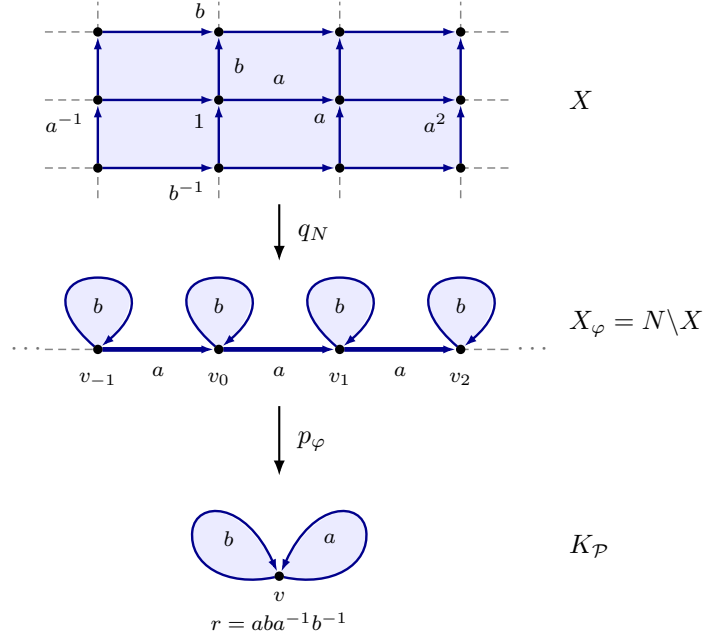

\Needspace{3\baselineskip}
The following result of Button turns infinite-dimensional \(H_1(N;\R)\) into a free quotient of a finite-index subgroup.

\begin{lemma}[{\cite[Corollary~2.2]{Button2008}}]
\label{lem:button}
Let \(\Gamma\) be a finitely presented group and let \(N\) be the kernel of an epimorphism \(\Gamma\twoheadrightarrow\Z\). If \(H_1(N;\R)\) is infinite-dimensional, then some finite-index subgroup of \(\Gamma\) maps onto the free group \(F_2\).
\end{lemma}

Button uses rational coefficients; the real-coefficient version is equivalent since \(H_1(N;\R)\cong H_1(N;\mathbb Q)\otimes_{\mathbb Q}\R\).

\begin{proof}[Proof of Theorem~\ref{thm:betti-two-gap}]
Suppose, for a contradiction, that \(\lambda_2^+(X)>0\), and fix \(0<\varepsilon<\lambda_2^+(X)\).

\medskip\noindent
\emph{Step 1. Write the boundary maps on the cyclic cover.}
Write \(\cP=\langle x_1,\ldots,x_s\mid r_1,\ldots,r_q\rangle\) and \(b=b_1(\Gamma)\ge2\). Choose an isomorphism \(\Gamma_{\mathrm{ab}}=\Gamma/[\Gamma,\Gamma]\cong\Z^b\oplus A\), with \(A\) finite, and let \(\varphi\) be the composite
\[
 \Gamma\twoheadrightarrow\Gamma_{\mathrm{ab}}
 \cong\Z^b\oplus A
 \xrightarrow{\operatorname{pr}_1}\Z,
\]
where \(\operatorname{pr}_1\) is the first coordinate projection. Thus \(\varphi\) is an epimorphism; set \(a_i:=\varphi(x_i)\in\Z\).

For the cover and notation of Definition~\ref{def:infinite-cyclic-cover}, let \(C_k(X_\varphi;\R)\) denote the space of finitely supported real cellular \(k\)-chains. A vertex chain \(h\in C_0(X_\varphi;\R)\) is identified with a finitely supported function \(h:V(X_\varphi)\to\R\). These chains are represented by Laurent polynomials, i.e. finite sums of integer powers of a formal variable \(t\), with multiplication by \(t\) corresponding to deck translation:
\[
 h\longmapsto h(t):=\sum_{n\in\Z}h(v_n)t^n\in\R[t,t^{-1}].
\]
An edge chain in \(C_1(X_\varphi;\R)\) is represented by the column vector \(u(t)=(u_i(t))_{i=1}^s\in\R[t,t^{-1}]^s\), where \(u_i(t)\) is the Laurent polynomial whose coefficient of \(t^n\) is the coefficient of the lifted \(x_i\)-edge based at \(v_n\). Similarly, a face chain in \(C_2(X_\varphi;\R)\) is represented by \(f(t)=(f_j(t))_{j=1}^q\in\R[t,t^{-1}]^q\), where \(f_j(t)\) records the coefficients of the lifted \(r_j\)-faces in the same way. Thus we obtain the chain complex
\[
 \R[t,t^{-1}]^q\xrightarrow{\partial_2(t)}
 \R[t,t^{-1}]^s\xrightarrow{\partial_1(t)}\R[t,t^{-1}],
 \qquad \partial_1(t)\partial_2(t)=0.
\]
The lifted \(x_i\)-edge at \(v_n\) has boundary \(v_{n+a_i}-v_n\). Thus, with \(t^n\) in the \(i\)-th component,
\[
 \begin{aligned}
 \partial_1(t)(0,\ldots,t^n,\ldots,0)^{\top}
 &=t^{n+a_i}-t^n=(t^{a_i}-1)t^n,\\
 \partial_1(t)u(t)&=\sum_{i=1}^s(t^{a_i}-1)u_i(t),\\
 \partial_1(t)&=(t^{a_1}-1,\ldots,t^{a_s}-1).
 \end{aligned}
\]
For the lifted \(r_j\)-face \(F_{n,j}\) based at \(v_n\) and the lifted \(x_i\)-edge \(e_{m,i}\) based at \(v_m\), write \(\eta(e_{m,i},F_{n,j})\) for their signed net incidence in \(X_\varphi\). Thus, with \(t^n\) in the \(j\)-th component,
\[
 \begin{aligned}
 \bigl(\partial_2(t)(0,\ldots,t^n,\ldots,0)^{\top}\bigr)_i
 &=\sum_{m\in\Z}\eta(e_{n+m,i},F_{n,j})\,t^{n+m}\\
 &=t^n\sum_{m\in\Z}\eta(e_{m,i},F_{0,j})\,t^m,\\
 \partial_2(t)_{ij}
 &=\sum_{g\in\Gamma}\eta(e(g,x_i),F(1,r_j))\,t^{\varphi(g)}.
 \end{aligned}
\]
Here the sums are finite, and the second equality holds as deck translation preserves incidence. We use the fact that
\[
 \eta(e_{m,i},F_{0,j})
 =\sum_{\substack{g\in\Gamma\\\varphi(g)=m}}
 \eta(e(g,x_i),F(1,r_j)),\qquad m\in\Z.
\]
Multiplying by \(t^m\) and summing over \(m\) gives the formula for \(\partial_2(t)_{ij}\) above.

Define \(\partial_k(1)\) by setting \(t=1\) in each matrix entry of \(\partial_k(t)\), which means that all deck translates of a given cell in \(X_\varphi\) are identified; hence
\[
 \partial_k(1)=\partial_k^{K_{\cP}},\qquad k=1,2.
\]
Since \(K_{\cP}\) has one vertex, \(\partial_1(1)=0\) and \(\ker\partial_1(1)=C_1(K_{\cP};\R)\cong\R^s\). Moreover, \(\pi_1(K_{\cP})=\Gamma\) gives \(\dim_\R H_1(K_{\cP};\R)=\dim_\R H_1(\Gamma;\R)=b\). We obtain
\[
 \begin{aligned}
 H_1(K_{\cP};\R)&\cong
 \ker\partial_1(1)/\ran\partial_2(1)
 \cong\R^s/\ran\partial_2(1),\\
 \operatorname{rank}_\R\partial_2(1)
 &=s-\dim_\R H_1(K_{\cP};\R)=s-b.
 \end{aligned}
\]

\medskip\noindent
\emph{Step 2. Use a positive face gap to fix the matrix rank.}
For the face Laplacian \(\Delta_2=\partial_2^*\partial_2\), the spectral theorem \cite[Theorem~IX.2.3]{Conway1990} gives
\[
 \spec(\Delta_2)\subseteq\{0\}\cup[\varepsilon,\infty)
 \quad\Longrightarrow\quad T:=\Delta_2^2-\varepsilon\Delta_2\ge0.
\]
Here \(\partial_2:\ell^2(\Gamma)^q\to\ell^2(\Gamma)^s\) is the bounded extension of the cellular boundary on \(X\). Set \(A_g=(\eta(e(g,x_i),F(1,r_j)))_{i,j}\in\R^{s\times q}\). These matrices vanish for all but finitely many \(g\). For a face chain \(f=(f_j)\), the coefficient of \(F(k,r_j)\) is \(f_j(k)\), and the coefficient of \(e(h,x_i)\) in its boundary is
\[
 \begin{aligned}
 (\partial_2f)_i(h)
 &=\sum_{j=1}^q\sum_{k\in\Gamma}
   \eta(e(h,x_i),F(k,r_j))\,f_j(k)\\
 &=\sum_{j=1}^q\sum_{g\in\Gamma}
   (A_g)_{ij}\,f_j(hg^{-1}).
 \end{aligned}
\]
The second equality follows from the substitution \(k=hg^{-1}\) and translation invariance. Since \(f_j(hg^{-1})=(R_{g^{-1}}f_j)(h)\), we obtain
\[
 \begin{aligned}
 \partial_2&=\sum_{g\in\Gamma}A_gR_{g^{-1}},\\
 \Delta_2=\partial_2^*\partial_2
 &=\sum_{g,h\in\Gamma}A_g^{\top}A_hR_{gh^{-1}}.
 \end{aligned}
\]
The second identity follows by expanding \(\partial_2^*\partial_2\) with \(R_{g^{-1}}^*=R_g\) and \(R_gR_{h^{-1}}=R_{gh^{-1}}\).
For \(z\in S^1\), set \(\chi_z(g)=z^{-\varphi(g)}\). Replacing \(R_{g^{-1}}\) with \(\chi_z(g^{-1})\) in the expression for \(\partial_2\) gives
\[
 \sum_{g\in\Gamma}A_g\chi_z(g^{-1})
 =\sum_{g\in\Gamma}A_gz^{\varphi(g)}=\partial_2(z).
\]
Here \(\partial_2(z)\) is the Laurent polynomial matrix \(\partial_2(t)\) evaluated at \(t=z\).
In the formula for \(\Delta_2\), we replace \(R_{gh^{-1}}\) by \(\chi_z(gh^{-1})=\overline{\chi_z(g^{-1})}\chi_z(h^{-1})\). The preceding formula for \(\partial_2(z)\) then gives
\[
 \sum_{g,h\in\Gamma}A_g^{\top}A_h
   \overline{\chi_z(g^{-1})}\chi_z(h^{-1})
 =\partial_2(z)^*\partial_2(z)=:B(z).
\]
The \(q\times q\) matrix \(B(z)\) is Hermitian and nonnegative.
The nonnegative operator \(T=\Delta_2^2-\varepsilon\Delta_2\) is a finite sum of translations with real \(q\times q\) matrix coefficients. This structure essentially comes from the \(\Gamma\)-periodicity and finite range of \(X^\vee\) established in Theorem~\ref{thm:dual-formula}. Since \(\Gamma\) is amenable, applying Lemma~\ref{lem:folner-evaluation} to \(T\) with \(\chi=\chi_z\) gives
\[
 B(z)^2-\varepsilon B(z)\ge0,\qquad z\in S^1.
\]
The ordered eigenvalues of the continuous Hermitian matrix \(B(z)\) vary continuously and lie in \(\{0\}\cup[\varepsilon,\infty)\). Since \(S^1\) is connected, the number of positive eigenvalues is constant. Hence
\[
 \operatorname{rank}_\C\partial_2(z)
 =\operatorname{rank}_\C B(z)
 =\operatorname{rank}_\R\partial_2(1)=s-b,\qquad z\in S^1.
\]

\medskip\noindent
\emph{Step 3. Use homology to contradict amenability.}
Every \((s-b+1)\)-minor of \(\partial_2(t)\) is a Laurent polynomial vanishing on \(S^1\) by Step 2; hence this polynomial is identically zero. We use \(\R(t)\), the field of real rational functions \(P(t)/Q(t)\), where \(P,Q\) are real polynomials, \(Q\ne0\), and \(t\) is a formal variable. The boundary matrices then act between finite-dimensional vector spaces, so their ranks determine the homology dimension. Thus,
\[
 s-b=\operatorname{rank}_\R\partial_2(1)
 \le\operatorname{rank}_{\R(t)}\partial_2(t)\le s-b.
\]
The left inequality follows because any minor nonzero at \(t=1\) remains nonzero over \(\R(t)\); the right inequality follows from the vanishing of all \((s-b+1)\)-minors.

Surjectivity of \(\varphi\) gives some \(a_i=\varphi(x_i)\ne0\), since the \(x_i\) generate \(\Gamma\). Thus \(t^{a_i}-1\) is a nonzero entry of the row matrix \(\partial_1(t)\) over \(\R(t)\), so \(\operatorname{rank}_{\R(t)}\partial_1(t)=1\). The chain condition \(\partial_1(t)\partial_2(t)=0\) therefore gives
\[
 \dim_{\R(t)}\bigl(\ker_{\R(t)}\partial_1(t)/\ran_{\R(t)}\partial_2(t)\bigr)
 =(s-1)-(s-b)=b-1>0.
\]

Choose a nonzero class \(\beta=[u]\) in this quotient, with representative \(u=(u_i)_{i=1}^s\). Write \(u_i=p_i(t)/q_i(t)\), where \(p_i,q_i\in\R[t]\) and \(q_i\ne0\), and set
\[
 D(t):=\prod_{i=1}^s q_i(t)\ne0,\qquad c:=D(t)u.
\]
Then
\[
 \begin{aligned}
 c_i&=p_i(t)\prod_{\substack{1\le j\le s\\j\ne i}}q_j(t)\in\R[t],
       \qquad 1\le i\le s,\\
 \partial_1(t)c&=D(t)\partial_1(t)u=0,\qquad [c]=D(t)\beta\ne0.
 \end{aligned}
\]
Here \([c]\) is the class in the same quotient; it is nonzero because \(D(t)\) is invertible in \(\R(t)\). Thus \(c\in\R[t,t^{-1}]^s\) is a finitely supported edge cycle with \(c\notin\ran_{\R(t)}\partial_2(t)\).

Multiplication by \(t^k\) gives the deck translate of \(c\), which is again a finitely supported edge cycle. We show that the classes \([t^kc]\), \(k\in\Z\), are linearly independent in \(H_1(X_\varphi;\R)\). Suppose that for some finite set \(K\subset\Z\) and real coefficients \(\alpha_k\), not all zero,
\[
 \sum_{k\in K}\alpha_k[t^kc]=0.
\]

Set \(p(t):=\sum_{k\in K}\alpha_kt^k\in\R[t,t^{-1}]\). Since the exponents are distinct and the coefficients are not all zero, \(p(t)\ne0\). The relation above says that the cycle \(p(t)c=\sum_{k\in K}\alpha_kt^kc\) represents the zero homology class. By the definition of cellular homology, this means that \(p(t)c\) is the boundary of a finitely supported face chain. Encoding this chain by a Laurent polynomial vector as in Step 1 gives \(f(t)\in\R[t,t^{-1}]^q\) with
\[
 p(t)c=\partial_2(t)f(t)
 \quad\Longrightarrow\quad
 c=\partial_2(t)\bigl(f(t)/p(t)\bigr).
\]
Since \(f(t)/p(t)\in\R(t)^q\), this contradicts \(c\notin\ran_{\R(t)}\partial_2(t)\). Hence the classes \([t^kc]\), \(k\in\Z\), are linearly independent over \(\R\). Since \(\pi_1(X_\varphi)=N\),
\[
 \dim_\R H_1(X_\varphi;\R)
 =\dim_\R H_1(N;\R)=\infty.
\]

Lemma~\ref{lem:button} now gives a finite-index subgroup \(\Gamma_0\le\Gamma\) with \(\Gamma_0\twoheadrightarrow F_2\). Since amenability passes to subgroups and quotients, this contradicts the nonamenability of \(F_2\). Thus the face gap vanishes, and Theorem~\ref{thm:spectral-transfer} gives
\[
 \lambda_2^+(X)=\lambda_{1,\mathrm{coex}}(X)=0.\qedhere
\]
\end{proof}

\begin{corollary}
\label{cor:abelian-gaps}
For \(d\ge2\), the Cayley \(2\)-complex \(X\) of any finite presentation of \(\Z^d\) satisfies \(\lambda_0(X)=\lambda_2^+(X)=0\).
\end{corollary}

\begin{proof}
The standard commutator presentation is
\[
 \Z^d=\langle x_1,\ldots,x_d\mid
 x_ix_jx_i^{-1}x_j^{-1},\ 1\le i<j\le d\rangle.
\]
In the presentation complex, each relator traverses the \(x_i\)- and \(x_j\)-edges once in each direction, so its cellular boundary is zero. Thus the face boundaries impose no linear relations among the \(d\) generator classes, giving \(H_1(\Z^d;\R)\cong\R^d\) and \(b_1(\Z^d)=d\ge2\). Since \(\Z^d\) is infinite and amenable, Theorem~\ref{thm:betti-two-gap} gives zero face gap, and Kesten's theorem gives zero vertex gap.
\end{proof}

The assumption \(b_1\ge2\) is used in Step~3 of the proof of Theorem~\ref{thm:betti-two-gap} to ensure \(b_1-1>0\). For amenable groups with \(b_1=1\), the face gap can be zero, as for \(\operatorname{BS}(1,m)\), \(m\ge2\), in Theorem~\ref{thm:bs-gap}, or positive, as for the presentation of the virtually cyclic group \(\Lambda\cong\Z\) in Proposition~\ref{ex:cyclic-gap}.

Products \(H\times\Z\) with \(H\) infinite and amenable give another vanishing result without the hypothesis \(b_1\ge2\).

\begin{proposition}
\label{prop:amenable-product}
Let \(H\) be an infinite amenable finitely presented group. For every finite presentation of \(H\times\Z\), its Cayley \(2\)-complex \(X\) satisfies \(\lambda_0(X)=\lambda_2^+(X)=0\).
\end{proposition}

\begin{proof}
Choose any finite presentation of \(H\), and let \(Y\) be its Cayley \(2\)-complex with edge boundary operator \(\partial_1\). For \(h\in\ell^2(H)\), the adjoint satisfies \((\partial_1^*h)(e(g,s))=h(gs)-h(g)\). If \(\partial_1^*h=0\), then \(h\) has equal values at the endpoints of every edge. The Cayley graph is connected, so \(h\) is constant on \(H\). Since \(H\) is infinite, a square-summable constant function must vanish; hence \(\ker\partial_1^*=\{0\}\).

Amenability gives zero spectral bottom for \(\partial_1\partial_1^*\) by Kesten's theorem. Since \(\ker\partial_1^*=\{0\}\), Theorem~\ref{thm:spectral-transfer} transfers this zero spectral bottom from the full vertex space to \(\partial_1^*\partial_1\) on \((\ker\partial_1)^\perp\). Thus choose real unit edge chains \(u_n\perp\ker\partial_1\) with \(\|\partial_1u_n\|_2\to0\).

Write the chosen presentation as \(H=\langle S\mid R\rangle\). Add a new generator \(t\) for the \(\Z\)-factor and impose \([s,t]=1\) for every \(s\in S\), making \(t\) commute with all generators of \(H\). This gives the product presentation
\[
 H\times\Z=\langle S\cup\{t\}\mid R,\ [s,t]\ (s\in S)\rangle.
\]
The Cayley \(2\)-complex is the \(2\)-skeleton \(X=(Y\times L)^{(2)}\), where \(L\) is the integer line with \(t\)-edges \([k,k+1]\). Let \(\delta_k\) denote the unit vertex chain at \(k\in\Z\) on \(L\). For an oriented edge \(e\) of \(Y\), the square \(e\times[k,k+1]\) has boundary
\begin{equation}
\label{eq:amenable-product-boundary}
 \partial_2(e\times[k,k+1])
 =(\partial_1e)\times[k,k+1]
 -e\times(\delta_{k+1}-\delta_k).
\end{equation}
Set
\[
 f_n=n^{-1/2}\sum_{k=0}^{n-1}u_n\times[k,k+1],\qquad \|f_n\|_2=1.
\]
For \(z\in\ker\partial_2\), restrict its coefficients to the \(k\)-th layer of squares \(e\times[k,k+1]\), identifying this restriction with an edge chain \(b_k\) on \(Y\) via \(b_k(e)=z(e\times[k,k+1])\). The vertical component of \(\partial_2z=0\) in \eqref{eq:amenable-product-boundary} gives \(\partial_1b_k=0\). Thus \(u_n\perp\ker\partial_1\) yields
\[
 \langle f_n,z\rangle
 =n^{-1/2}\sum_{k=0}^{n-1}\langle u_n,b_k\rangle=0,
\]
so \(f_n\perp\ker\partial_2\). In the horizontal boundary sum from \eqref{eq:amenable-product-boundary}, contributions at intermediate levels cancel:
\[
 -n^{-1/2}u_n\times\sum_{k=0}^{n-1}(\delta_{k+1}-\delta_k)
 =n^{-1/2}u_n\times(\delta_0-\delta_n).
\]
The two horizontal end terms and the \(n\) vertical terms on distinct intervals have disjoint edge supports. For the unit chain \(f_n\perp\ker\partial_2\), the Rayleigh quotient gives
\[
\begin{aligned}
 0\le\lambda_2^+(X)
 &\le\|\partial_2f_n\|_2^2\\
 &=\frac2n\|u_n\|_2^2
   +\frac1n\sum_{k=0}^{n-1}\|\partial_1u_n\|_2^2\\
 &=\frac2n+\|\partial_1u_n\|_2^2\longrightarrow0.
\end{aligned}
\]
Proposition~\ref{prop:presentation-independence} extends the face-gap conclusion to every finite presentation of \(H\times\Z\). Kesten's theorem gives \(\lambda_0(X)=0\) since \(H\times\Z\) is amenable.
\end{proof}

Since \(b_1(H\times\Z)=b_1(H)+1\), Theorem~\ref{thm:betti-two-gap} already covers the case \(b_1(H)\ge1\). Proposition~\ref{prop:amenable-product} adds the case \(b_1(H)=0\). For example, take the infinite dihedral group \(H=\langle r,s\mid r^2,s^2\rangle\). Put \(a=rs\). Then \(H=\langle a\rangle\sqcup\langle a\rangle r\), so \([H:\langle a\rangle]=2\).

\(H\) is an infinite amenable finitely presented group. Since its presentation complex has face boundary matrix \(2I_2\), we have \(b_1(H)=0\). The product \(H\times\Z\) has a subgroup of index two isomorphic to \(\Z^2\), so it is not virtually cyclic. By Proposition~\ref{prop:amenable-product}, the Cayley \(2\)-complex \(X\) of any finite presentation of \(H\times\Z\) satisfies \(\lambda_0(X)=\lambda_2^+(X)=0\).

\section{Examples and the conjecture}
\label{sec:examples}

In Sections~\ref{sec:bs} and~\ref{sec:amenable}, we saw examples with zero vertex gap and zero face gap. We now give an example with positive vertex gap and zero face gap. This suggests that, outside the virtually cyclic case, vanishing of the face gap may be a weaker condition than vanishing of the vertex gap.

\begin{proposition}
\label{prop:one-sided-gap}
Let \(F_r\) be the free group of rank \(r\ge2\).  The Cayley \(2\)-complex of the standard presentation
\[
 F_r\times\Z
 =\langle a_1,\ldots,a_r,t\mid[a_1,t],\ldots,[a_r,t]\rangle
\]
satisfies \(\ker\partial_2=\{0\}\) and \(\lambda_2^+(X)=0<\lambda_0(X)\). There are real unit face chains \(f_N\) with \(\|\partial_2f_N\|_2^2=2/N\).
\end{proposition}

\begin{proof}
The complex is \(X=T_{2r}\times L\), where \(T_{2r}\) is the \(2r\)-regular tree and \(L\) is the line with oriented edges \([k,k+1]\), \(k\in\Z\). Its faces are the squares \(e\times[k,k+1]\). Write a face chain as \(w=(w_k)_{k\in\Z}\), with \(w_k\in C_1^{(2)}(T_{2r})\). Here \(w_k(e)\) is the coefficient of the face \(e\times[k,k+1]\) in \(w\). Then \(\|w\|_2^2=\sum_k\|w_k\|_2^2\).

Let \(\partial_T\) be the edge boundary on the tree. A square has boundary
\[
 \partial_2(e\times[k,k+1])
 =(\partial_Te)\times[k,k+1]-e\times(\delta_{k+1}-\delta_k).
\]
The vertical and horizontal edges are disjoint, so
\begin{equation}
\label{eq:product-boundary-norm}
 \|\partial_2w\|_2^2
 =\sum_k\|\partial_Tw_k\|_2^2+
   \sum_k\|w_k-w_{k-1}\|_2^2.
\end{equation}
If \(\partial_2w=0\), the second sum gives \(w_k=w_{k-1}\) for every \(k\). Square summability then forces \(w=0\), so \(\partial_2\) is injective.

We next construct a nonzero square-summable tree flow with zero boundary. Orient the tree away from a root and assign coefficients \(1\) and \(-1\) to two edges leaving the root, with zero on the other branches. At each later vertex, divide the incoming coefficient equally among the \(2r-1\) outgoing edges. The coefficients balance at every vertex. Its squared norm is
\[
 2\sum_{j\ge0}(2r-1)^j(2r-1)^{-2j}
 =2\sum_{j\ge0}(2r-1)^{-j}<\infty.
\]
Normalize this flow to obtain a real unit vector \(u\in\ker\partial_T\). Set
\[
 (f_N)_k=N^{-1/2}\mathbf1_{\{0,\ldots,N-1\}}(k)u.
\]
Since \(\partial_Tu=0\), equation~\eqref{eq:product-boundary-norm} gives
\[
 \|f_N\|_2^2=1,\qquad
 \|\partial_2f_N\|_2^2=\frac2N,\qquad
 0\le\lambda_2^+(X)\le\frac2N\longrightarrow0.
\]
Finally, \(F_r\times\Z\) contains the nonamenable subgroup \(F_r\), so Kesten's theorem gives \(\lambda_0(X)>0\) \cite{KestenMean1959}.
\end{proof}

The next example explains the virtually cyclic exception in Conjecture~\ref{conj:coexact-implies-vertex}.

\begin{proposition}
\label{ex:cyclic-gap}
The Cayley \(2\)-complex of
\[
 \Lambda=\langle a,b\mid ba^{-1}\rangle\cong\Z
\]
has \(\lambda_0(X)=0\) and \(\lambda_2^+(X)=2\).
\end{proposition}

\begin{proof}
Although \(a\) and \(b\) represent the same element of \(\Lambda\), they label distinct edges. The face at \(g\) has boundary \(e(g,b)-e(g,a)\), so
\[
 \partial_2f=(-f,f),\qquad
 \partial_2^*\partial_2=2I.
\]
The vertex graph is a line with doubled edges. Taking \(f_N=\mathbf1_{\{0,\ldots,N-1\}}\) gives
\[
 \|f_N\|_2^2=N,\qquad
 \|\partial_1^*f_N\|_2^2=4,\qquad
 0\le\lambda_0(X)\le\frac4N\longrightarrow0.
\]
\end{proof}

For the standard presentation \(\Z=\langle a\mid\varnothing\rangle\), the Cayley graph has no faces, so the positive face spectrum is empty and the face gap is \(+\infty\) by convention. The presentation above has face gap \(2\). Both gaps are positive, illustrating Proposition~\ref{prop:presentation-independence}.

These examples motivate Conjecture~\ref{conj:coexact-implies-vertex}: for an infinite finitely presented group that is not virtually cyclic, a positive face gap should imply a positive vertex gap.

\section*{Acknowledgement}

B.~Hua is supported by NSFC, no.12371056.
Y.~Lin is supported by NSFC, nos.12071245 and 12471088.
T.~Wang is partially supported by the National Key R\&D Program of China, no.~2025YFA1017500.
H.~Zhang was partially funded by the Research Council of Finland Center of Excellence FiRST (project number~364217).
The author B.~Hua is grateful to Max Planck Institute for Mathematics in Bonn
for its hospitality and financial support.

% Leave room for the additional affiliation in the final address block.
\enlargethispage{2\baselineskip}
\end{document}